\documentclass[11pt]{amsart}

\usepackage{
    amsmath,
    amsfonts,
    amssymb,
    amsthm,
    amscd,
    comment,
    enumitem,
    etoolbox,
    gensymb,    
    mathtools,
    mathdots,
    stmaryrd,
}
\usepackage[usenames,dvipsnames]{xcolor}
\usepackage[all]{xy}

\usepackage[T1]{fontenc}
\usepackage{bbm}                     
\usepackage[colorlinks=true, linkcolor=blue, citecolor=blue, urlcolor=blue, breaklinks=true]{hyperref}

\DeclareFontFamily{OT1}{pzc}{}
\DeclareFontShape{OT1}{pzc}{m}{it}{<-> s * [1.10] pzcmi7t}{}
\DeclareMathAlphabet{\mathpzc}{OT1}{pzc}{m}{it}

\usepackage[capitalize]{cleveref}   

\crefname{defin}{Definition}{Definitions}
\crefname{eg}{Example}{Examples}
\crefname{egs}{Example}{Examples}
\crefname{lem}{Lemma}{Lemmas}
\crefname{theo}{Theorem}{Theorems}
\crefname{equation}{}{}
\crefname{enumi}{}{}

\newcommand\C{\mathbb{C}}
\newcommand\N{\mathbb{N}}
\newcommand\OO{\mathbb{O}}

\newcommand\R{\mathbb{R}}
\newcommand\Z{\mathbb{Z}}
\newcommand\kk{\Bbbk}
\newcommand\one{\mathbbm{1}}

\newcommand\B{\mathbf{B}}

\newcommand\fg{\mathfrak{g}}
\newcommand\gl{\mathfrak{gl}}           
\newcommand\fS{\mathfrak{S}}            

\newcommand\md{\textup{-mod}}

\newcommand\cC{\mathcal{C}}

\newcommand\Fcat{\mathpzc{F}}           
\newcommand\OB{\mathpzc{OB}}            
\newcommand\Tcat{\mathpzc{T}}           
\newcommand\Wcat{\mathpzc{W}}           

\newcommand\cI{\mathcal{I}}             

\newcommand\go{{\mathsf{I}}}            

\DeclareMathOperator{\Add}{Add}
\DeclareMathOperator{\End}{End}

\DeclareMathOperator{\Hom}{Hom}

\DeclareMathOperator{\Kar}{Kar}
\DeclareMathOperator{\Mat}{Mat}

\DeclareMathOperator{\Ob}{Ob}

\DeclareMathOperator{\Rot}{Rot}
\DeclareMathOperator{\RP}{Re}       
\DeclareMathOperator{\Span}{span}
\DeclareMathOperator{\Switch}{Switch}

\DeclareMathOperator{\Tr}{Tr}
\DeclareMathOperator{\tr}{tr}

\usepackage{tikz}
\usetikzlibrary{arrows.meta}
\usetikzlibrary{decorations.markings}
\usetikzlibrary{calc}

\tikzset{anchorbase/.style={>=To,baseline={([yshift=-0.5ex]current bounding box.center)}}}
\tikzset{ 
    centerzero/.style={>=To,baseline={([yshift=-0.5ex](#1))}},
    centerzero/.default={0,0}
}
\tikzset{wipe/.style={white,line width=3pt}}

\newcommand\braidup{to[out=up,in=down]}
\newcommand\braiddown{to[out=down,in=up]}

\newcommand\opendot[1]{
    \filldraw[fill=white, draw=black] (#1) circle (0.05);
}

\newcommand\bub[1]{
    \draw (#1)++(0,0.2) arc(90:-270:0.2)
}

\newcommand\bubble{
    \begin{tikzpicture}[centerzero]
        \bub{0,0};
    \end{tikzpicture}
}
\newcommand\idstrand[1][a]{
    \begin{tikzpicture}[centerzero]
        \draw (0,-0.2) -- (0,0.2);
    \end{tikzpicture}
}
\newcommand\crossmor{
    \begin{tikzpicture}[centerzero]
        \draw (0.2,-0.2) -- (-0.2,0.2);
        \draw (-0.2,-0.2) -- (0.2,0.2);
    \end{tikzpicture}
}
\newcommand{\cupmor}{
    \begin{tikzpicture}[anchorbase]
        \draw (-0.15,0.15) -- (-0.15,0) arc(180:360:0.15) -- (0.15,0.15);
    \end{tikzpicture}
}
\newcommand{\capmor}{
    \begin{tikzpicture}[anchorbase]
        \draw (-0.15,-0.15) -- (-0.15,0) arc(180:0:0.15) -- (0.15,-0.15);
    \end{tikzpicture}
}
\newcommand\mergemor{
    \begin{tikzpicture}[centerzero]
      \draw (-0.2,-0.2) to (0,0);
      \draw (0.2,-0.2) to (0,0);
      \draw (0,0) to (0,0.2);
    \end{tikzpicture}
}
\newcommand\splitmor{
    \begin{tikzpicture}[centerzero]
      \draw (-0.2,0.2) to (0,0);
      \draw (0.2,0.2) to (0,0);
      \draw (0,0) to (0,-0.2);
    \end{tikzpicture}
}

\newcommand\lollydrop{
    \begin{tikzpicture}[centerzero]
        \draw (0,0.2) -- (0,0) to[out=-40,in=up] (0.15,-0.15) arc(0:-180:0.15) to[out=up,in=-130] (0,0);
    \end{tikzpicture}
}
\newcommand\dotcross{
    \begin{tikzpicture}[centerzero]
        \draw (0.2,-0.2) -- (-0.2,0.2);
        \draw (-0.2,-0.2) -- (0.2,0.2);
        \opendot{0,0};
    \end{tikzpicture}
}
\newcommand\triform{
    \begin{tikzpicture}[centerzero]
        \draw (-0.3,-0.2) -- (0,0.2);
        \draw (0,-0.2) -- (0,0.2);
        \draw (0.3,-0.2) -- (0,0.2);
    \end{tikzpicture}
}
\newcommand\explode{
    \begin{tikzpicture}[centerzero]
        \draw (-0.3,0.2) -- (0,-0.2);
        \draw (0,0.2) -- (0,-0.2);
        \draw (0.3,0.2) -- (0,-0.2);
    \end{tikzpicture}
}
\newcommand\trimor{
    \begin{tikzpicture}[anchorbase]
        \draw (0,0.15) -- (0.13,-0.075) -- (-0.13,-0.075) -- cycle;
        \draw (0,0.15) -- (0,0.3);
        \draw (0.13,-0.075) -- (0.24,-0.185);
        \draw (-0.13,-0.075) -- (-0.24,-0.185);
    \end{tikzpicture}
}

\newcommand\sqmor{
    \begin{tikzpicture}[centerzero]
        \draw (-0.15,-0.15) rectangle (0.15,0.15);
        \draw (-0.3,-0.3) -- (-0.15,-0.15);
        \draw (0.3,-0.3) -- (0.15,-0.15);
        \draw (-0.3,0.3) -- (-0.15,0.15);
        \draw (0.3,0.3) -- (0.15,0.15);
    \end{tikzpicture}
}
\newcommand\pentmor{
    \begin{tikzpicture}[centerzero]
        \draw (0,0.25) -- (-0.24,0.08) -- (-0.15,-0.20) -- (0.15,-0.20) -- (0.24,0.08) -- cycle;
        \draw (0,0.25) -- (0,0.4);
        \draw (-0.24,0.08) -- (-0.35,0.4);
        \draw (0.24,0.08) -- (0.35,0.4);
        \draw (-0.15,-0.2) -- (-0.35,-0.4);
        \draw (0.15,-0.2) -- (0.35,-0.4);
    \end{tikzpicture}
}
\newcommand\Hmor{
    \begin{tikzpicture}[centerzero]
        \draw (-0.3,-0.2) -- (-0.1,0) -- (-0.3,0.2);
        \draw (-0.1,0) -- (0.1,0);
        \draw (0.3,-0.2) -- (0.1,0) -- (0.3,0.2);
    \end{tikzpicture}
}
\newcommand\Imor{
    \begin{tikzpicture}[centerzero]
        \draw (-0.2,-0.3) -- (0,-0.1) -- (0.2,-0.3);
        \draw (-0.2,0.3) -- (0,0.1) -- (0.2,0.3);
        \draw (0,-0.1) -- (0,0.1);
    \end{tikzpicture}
}
\newcommand\hourglass{
    \begin{tikzpicture}[centerzero]
        \draw (-0.15,-0.3) -- (-0.15,-0.25) arc(180:0:0.15) -- (0.15,-0.3);
        \draw (-0.15,0.3) -- (-0.15,0.25) arc(180:360:0.15) -- (0.15,0.3);
    \end{tikzpicture}
}
\newcommand\jail{
    \begin{tikzpicture}[centerzero]
        \draw (-0.15,-0.3) -- (-0.15,0.3);
        \draw (0.15,-0.3) -- (0.15,0.3);
    \end{tikzpicture}
}

\newcommand{\symbox}[2]{
    \filldraw[fill=white,draw=black] (#1) rectangle (#2)
}
\newcommand{\antbox}[2]{
    \filldraw[black] (#1) rectangle (#2)
}

\newtheorem{theo}{Theorem}[section]

\newtheorem{prop}[theo]{Proposition}
\newtheorem{lem}[theo]{Lemma}
\newtheorem{cor}[theo]{Corollary}
\newtheorem{conj}[theo]{Conjecture}

\theoremstyle{definition}
\newtheorem{defin}[theo]{Definition}
\newtheorem{rem}[theo]{Remark}

\numberwithin{equation}{section}
\allowdisplaybreaks

\setenumerate[1]{label=(\alph*)}          

\newtoggle{comments}
\newtoggle{details}
\newtoggle{detailsnote}

\iftoggle{comments}{%
  \usepackage[notref,notcite]{showkeys}   
  \newcommand{\acomments}[1]{
    \ \\
    {\color{red}
      \textbf{AS:} #1
    }
    \ \\
    }
  \newcommand{\raj}[1]{
    \ \\
    {\color{purple}
      \textbf{For Raj:} #1
    }
    \ \\
    }
}{%
  \newcommand{\acomments}[1]{}
  \newcommand{\raj}[1]{}
}

\iftoggle{details}{%
  \newcommand{\details}[1]{
      \ \\
      {\color{OliveGreen}
        \textbf{Details:} #1
      }
      \\
  }
}{%
  \newcommand{\details}[1]{}
}

\begin{document}

\title{Diagrammatics for $F_4$}

\author{Raj Gandhi}
\address[R.G.]{
  Department of Mathematics and Statistics \\
  University of Ottawa \\
  Ottawa, ON K1N 6N5, Canada
}
\email{rgand037@uottawa.ca}

\author{Alistair Savage}
\address[A.S.]{
  Department of Mathematics and Statistics \\
  University of Ottawa \\
  Ottawa, ON K1N 6N5, Canada
}
\urladdr{\href{https://alistairsavage.ca}{alistairsavage.ca}, \textrm{\textit{ORCiD}:} \href{https://orcid.org/0000-0002-2859-0239}{orcid.org/0000-0002-2859-0239}}
\email{alistair.savage@uottawa.ca}

\author{Kirill Zainoulline}
\address[K.Z.]{
  Department of Mathematics and Statistics \\
  University of Ottawa \\
  Ottawa, ON K1N 6N5, Canada
}
\urladdr{\href{https://mysite.science.uottawa.ca/kzaynull/}{mysite.science.uottawa.ca/kzaynull/}}
\email{kzaynull@uottawa.ca}

\begin{abstract}
    We define a diagrammatic monoidal category, together with a full and essentially surjective monoidal functor from this category to the category of modules over the exceptional Lie algebra of type $F_4$.  In this way, we obtain a set of diagrammatic tools for studying type $F_4$ representation theory analogous to those of the oriented and unoriented Brauer categories in classical type.
\end{abstract}



\ifboolexpr{togl{comments} or togl{details}}{%
  {\color{magenta}DETAILS OR COMMENTS ON}
}{%
}

\maketitle
\thispagestyle{empty}


\section{Introduction}

If $V = \C^d$ denotes the natural module for the complex general linear Lie algebra $\gl_d$, Schur--Weyl duality states that the natural actions of $\gl_d$ and the symmetric group $\fS_k$ on $V^{\otimes k}$ commute and generate each other's centralizers.  This classic result can be extended to the definition of a full monoidal functor from the diagrammatic \emph{oriented Brauer category} $\OB_d$ of \cite{BCNR17} to the category of $\gl_d$-modules.  (Throughout the paper we consider \emph{finite-dimensional} modules.)  This functor sends the two generating objects $\uparrow$ and $\downarrow$ of $\OB_d$ to $V$ and its dual $V^*$.  Since \linebreak $\Hom_{\OB_d}(\uparrow^{\otimes k}) \cong \C \fS_k$, we have an induced surjective algebra homomorphism
\[
    \C \fS_k \twoheadrightarrow \End_{\gl_d}(V^{\otimes k}),
\]
recovering one of the statements of Schur--Weyl duality.  After passing to the additive Karoubi envelope $\Kar(\OB_d)$ of $\OB_d$, the above functor is also essentially surjective.  It follows that $\gl_d$-mod is a quotient of $\Kar(\OB_d)$ by a tensor ideal.  This observation allows one to use powerful and intuitive diagrammatic techniques in the study of the representation theory of the general linear Lie algebra.  For example, this approach leads to the definition of Deligne's interpolating categories \cite{Del07}.

Analogues of the above theory also exist in types $BCD$.  Here the oriented Brauer category is replaced by the \emph{Brauer category} of \cite{LZ15}, whose endomorphism algebras are \emph{Brauer algebras}.  However, analogous techniques in \emph{exceptional} type are not as well developed.  For type $G_2$, the diagrammatic category has been described by Kuperberg \cite{Kup94,Kup96}.  (In fact, Kuperberg treats the quantum case; see below.)  Invariant tensors for classical and exceptional semisimple Lie algebras have been computed diagrammatically by Cvitanovi\'{c} \cite{Cvi08}, but the approach there is rather different, inspired by the language of Feynman diagrams in quantum field theory.  This approach has been further investigated for exceptional Lie algebras in \cite{MT07,Wen03,Wes03}.

The goal of the current paper is to develop a diagrammatic category for type $F_4$ analogous to the oriented and unoriented Brauer categories in classical type.  Hints of the defining relations appear in the aforementioned papers.  In particular, several of the equations deduced in the current paper can be found in \cite[Ch.~19]{Cvi08} and \cite{Thu04} in a different language.  However, a complete treatment from the monoidal category point of view seems to be new.

Given a field $\kk$ of characteristic zero, we define a strict monoidal $\kk$-linear category $\Fcat = \Fcat_{\alpha,\delta}$, depending on two parameters $\alpha,\delta \in \kk$.  (In fact, up to isomorphism, the category is independent of $\alpha$; see \cref{whyalpha}.)  We consider the strict $\kk$-linear monoidal category generated by a single object $\go$ and four morphisms
\[
    \mergemor \colon \go^{\otimes 2} \to \go,\qquad
    \crossmor \colon \go^{\otimes 2} \to \go^{\otimes 2},\qquad
    \cupmor \colon \one \to \go^{\otimes 2},\qquad
    \capmor \colon \go^{\otimes 2} \to \one,
\]
where $\one$ is the unit object.  These morphisms are subject to certain relations, which we split into two families.   We denote by $\Tcat = \Tcat_{\alpha,\delta}$ the category obtained by imposing the first family of relations; see \cref{Tdef}.  These relations imply, in particular, that the category is symmetric monoidal and strict pivotal, the trivalent vertex is symmetric, and the generating object $\go$ is symmetrically self-dual of categorical dimension $\delta$.  To obtain the category $\Fcat$, we then impose three additional relations; see \cref{Fdef}.  The first of these is the relation
\begin{equation} \label{CH}
    \Hmor + \Imor +
    \begin{tikzpicture}[centerzero]
        \draw (-0.2,-0.4) -- (-0.2,0.4);
        \draw (-0.2,-0.2) -- (0.4,0.4);
        \draw (0.4,-0.4) -- (-0.2,0.2);
    \end{tikzpicture}
    = \frac{2\alpha}{\delta+2}
    \left(\, \jail + \hourglass + \crossmor\, \right),
\end{equation}
while the other two express the square and pentagon in terms of acyclic diagrams.

When $\kk=\C$, we define (\cref{magneto,baja}) a monoidal functor
\[
    \Phi \colon \Fcat_{7/3,26} \to \fg\md,
\]
where $\fg$ is the complex simple Lie algebra of type $F_4$.  The generating object $\go$ is sent to the natural $\fg$-module $V$, which is $26$-dimensional.  The compact Lie group $G$ corresponding to $\fg$ is the group of algebra automorphisms of the Albert algebra which, over the complex numbers, is the unique exceptional Jordan algebra.  The natural module $V$ can be identified with the traceless part of the Albert algebra.  Multiplication in the Albert algebra gives rise to a $\fg$-module homomorphism $V^{\otimes 2} \to V$, which is the image under $\Phi$ of the trivalent vertex $\mergemor$.  The morphism $\crossmor$ corresponds to the symmetric braiding in $\fg$-mod, and $\capmor$ is sent to the natural invariant bilinear form on $V$ coming from the trace on the Albert algebra.  The morphism $\cupmor$ is also determined by this bilinear form.  In this way, the category $\Fcat$ can also be viewed as a diagrammatic calculus for the Albert algebra.  The relation \cref{CH} corresponds to the Cayley--Hamilton theorem for $V$; see \cref{prestige}.

The functor $\Phi$ is defined only when $\delta=26$, since that is the dimension of the natural $\fg$-module $V$.  However, the diagrammatic category $\Fcat_{\alpha,\delta}$ is defined for any $\delta \ne -2$.  (When $\delta=-2$, the preliminary category $\Tcat$ collapses to the trivial category.)  The importance of the case $\delta=26$ can be seen purely diagrammatically.  It corresponds to the fact that $V$ is not a summand of the tensor square of the first fundamental representation of $\fg$.  See \cref{sack}.

Since the category $\fg$-mod is idempotent complete, the functor $\Phi$ induces a functor
\[
    \Kar(\Phi) \colon \Kar(\Fcat_{7/3,26}) \to \fg\md.
\]
Then $\Kar(\Phi)$ is full since $\Phi$ is.  In addition, we show (\cref{splat}) that $\Kar(\Phi)$ is essentially surjective.  Thus $\fg\md$ is equivalent to a quotient of the diagrammatic category $\Fcat$ by a tensor ideal.  In fact, we conjecture that $\Kar(\Phi)$ is also faithful, and hence an equivalence of categories; see \cref{faithful,hesitant}.  We also give (\cref{pipe}) conjectural spanning sets for the morphism spaces in $\Fcat$.

One immediate consequence of the fullness of $\Phi$ (\cref{SW}) is that we have surjective algebra homomorphisms
\[
    \End_{\Fcat}(\go^{\otimes k}) \twoheadrightarrow \End_\fg(V^{\otimes k}),\quad k \in \N.
\]
In other words, the endomorphism algebras in $\Fcat$ play the role in type $F_4$ that the oriented and unoriented Brauer algebras play in the classical types.

In classical types, quantum versions of the relevant diagrammatic categories exist.  In type $A$, the quantum analogue of the oriented Brauer category is the framed HOMFLYPT skein category.  In types $BCD$, the analogue of the unoriented Brauer category is the Kauffman skein category.  As their names suggest, both categories are closely related to important knot invariants.  In type $G_2$, the connection to the corresponding Reshetikhin--Turaev invariant is discussed in \cite{Kup94,Kup96}.  There should also exist a quantum analogue of the category $\Fcat$, related to the Reshetikhin--Turaev invariant in type $F_4$.  These quantum diagrammatics should also be related to a quantum version of the Albert algebra.

The category $\Fcat$ is also a first step towards a category of webs for type $F_4$.  The main goal in the theory of webs is to give a presentation, in terms of generators and relations, for the full monoidal subcategory of the category of modules for a quantized enveloping algebra, generated by the fundamental modules.  Such presentations are typically in terms of diagrammatic categories known as \emph{web categories}.  Web categories were first developed for rank two simple complex Lie algebras by Kuperberg \cite{Kup96}.  Then, in more general type $A$, they were described by Cautis--Kamnitzer--Morrison \cite{CKM14}.  More recently, the type $C$ case has been treated in \cite{BERT21}; see also \cite{Wes08}.  The degenerate (that is, $q=1$) web category for type $F_4$ should be the full monoidal subcategory of $\Kar(\Fcat)$ generated by objects corresponding to the four fundamental modules.  We explicitly identify three of these objects in \cref{sec:fundamental}.

\subsection*{Acknowledgements}

The research of R.G.\ was supported by an Ontario Graduate Scholarship and a Canada Graduate Scholarship from the Natural Sciences and Engineering Research Council of Canada (NSERC).  The research of A.S.\ and K.Z.\ was supported by NSERC Discovery Grants RGPIN-2017-03854 and RGPIN-2015-04469, respectively.  We thank Erhard Neher for helpful conversations and Bruce Westbury for useful remarks on an earlier version of this paper.  We are also grateful to the referees for comments that helped to improve the paper.

\section{The diagrammatic category\label{sec:defin}}

We fix a field $\kk$ of characteristic zero.  All categories are $\kk$-linear and all algebras and tensor products are over $\kk$ unless otherwise specified.  We let $\one$ denote the unit object of a monoidal category.  For objects $X$ and $Y$ in a category $\cC$, we denote by $\cC(X,Y)$ the vector space of morphisms from $X$ to $Y$.

\begin{defin} \label{Tdef}
    Fix $\alpha \in \kk^\times$ and $\delta \in \kk$.  Let $\Tcat = \Tcat_{\alpha,\delta}$ be the strict monoidal category generated by the object $\go$ and generating morphisms
    \begin{equation} \label{lego}
        \mergemor \colon \go \otimes \go \to \go,\quad
        \crossmor \colon \go \otimes \go \to \go \otimes \go,\quad
        \cupmor \colon \one \to \go \otimes \go,\quad
        \capmor \colon \go \otimes \go \to \one,
    \end{equation}
    subject to the following relations:
    \begin{gather} \label{vortex}
        \begin{tikzpicture}[centerzero]
            \draw (-0.3,-0.4) -- (-0.3,0) arc(180:0:0.15) arc(180:360:0.15) -- (0.3,0.4);
        \end{tikzpicture}
        =
        \begin{tikzpicture}[centerzero]
            \draw (0,-0.4) -- (0,0.4);
        \end{tikzpicture}
        =
        \begin{tikzpicture}[centerzero]
            \draw (-0.3,0.4) -- (-0.3,0) arc(180:360:0.15) arc(180:0:0.15) -- (0.3,-0.4);
        \end{tikzpicture}
        \ ,\quad
        \splitmor
        :=
        \begin{tikzpicture}[anchorbase]
            \draw (-0.4,0.2) to[out=down,in=180] (-0.2,-0.2) to[out=0,in=225] (0,0);
            \draw (0,0) -- (0,0.2);
            \draw (0.3,-0.3) -- (0,0);
        \end{tikzpicture}
        =
        \begin{tikzpicture}[anchorbase]
            \draw (0.4,0.2) to[out=down,in=0] (0.2,-0.2) to[out=180,in=-45] (0,0);
            \draw (0,0) -- (0,0.2);
            \draw (-0.3,-0.3) -- (0,0);
        \end{tikzpicture}
        \ ,\quad
        \begin{tikzpicture}[centerzero]
            \draw (-0.2,-0.3) -- (-0.2,-0.1) arc(180:0:0.2) -- (0.2,-0.3);
            \draw (-0.3,0.3) \braiddown (0,-0.3);
        \end{tikzpicture}
        =
        \begin{tikzpicture}[centerzero]
            \draw (-0.2,-0.3) -- (-0.2,-0.1) arc(180:0:0.2) -- (0.2,-0.3);
            \draw (0.3,0.3) \braiddown (0,-0.3);
        \end{tikzpicture}
        \ ,
        \\ \label{venom}
        \begin{tikzpicture}[centerzero]
            \draw (0.2,-0.4) to[out=135,in=down] (-0.15,0) to[out=up,in=-135] (0.2,0.4);
            \draw (-0.2,-0.4) to[out=45,in=down] (0.15,0) to[out=up,in=-45] (-0.2,0.4);
        \end{tikzpicture}
        =
        \begin{tikzpicture}[centerzero]
            \draw (-0.15,-0.4) -- (-0.15,0.4);
            \draw (0.15,-0.4) -- (0.15,0.4);
        \end{tikzpicture}
        \ ,\quad
        \begin{tikzpicture}[centerzero]
            \draw (0.3,-0.4) -- (-0.3,0.4);
            \draw (0,-0.4) to[out=135,in=down] (-0.25,0) to[out=up,in=-135] (0,0.4);
            \draw (-0.3,-0.4) -- (0.3,0.4);
        \end{tikzpicture}
        =
        \begin{tikzpicture}[centerzero]
            \draw (0.3,-0.4) -- (-0.3,0.4);
            \draw (0,-0.4) to[out=45,in=down] (0.25,0) to[out=up,in=-45] (0,0.4);
            \draw (-0.3,-0.4) -- (0.3,0.4);
        \end{tikzpicture}
        \ ,\quad
        \begin{tikzpicture}[anchorbase,scale=0.8]
            \draw (-0.4,-0.5) -- (0.2,0.3) -- (0.4,0.1) -- (0,-0.5);
            \draw (0.2,0.3) -- (0.2,0.6);
            \draw (-0.4,0.6) -- (-0.4,0.1) -- (0.4,-0.5);
        \end{tikzpicture}
        =
        \begin{tikzpicture}[anchorbase]
            \draw (-0.4,-0.4) -- (-0.2,-0.2) -- (-0.2,0) -- (0.2,0.4);
            \draw (0,-0.4) -- (-0.2,-0.2);
            \draw (0.2,-0.4) -- (0.2,0) -- (-0.2,0.4);
        \end{tikzpicture}
        \ ,\quad
        \begin{tikzpicture}[anchorbase,scale=0.8]
            \draw (-0.4,0.5) -- (0.2,-0.3) -- (0.4,-0.1) -- (0,0.5);
            \draw (0.2,-0.3) -- (0.2,-0.6);
            \draw (-0.4,-0.6) -- (-0.4,-0.1) -- (0.4,0.5);
        \end{tikzpicture}
        =
        \begin{tikzpicture}[anchorbase]
            \draw (-0.4,0.4) -- (-0.2,0.2) -- (-0.2,0) -- (0.2,-0.4);
            \draw (0,0.4) -- (-0.2,0.2);
            \draw (0.2,0.4) -- (0.2,0) -- (-0.2,-0.4);
        \end{tikzpicture}
        \ ,
        \\ \label{chess}
        \begin{tikzpicture}[anchorbase]
            \draw (-0.15,-0.4) to[out=45,in=down] (0.15,0) arc(0:180:0.15) to[out=down,in=135] (0.15,-0.4);
        \end{tikzpicture}
        = \capmor
        \ ,\quad
        \begin{tikzpicture}[anchorbase]
            \draw (-0.2,-0.5) to[out=45,in=down] (0.15,-0.2) to[out=up,in=-45] (0,0) -- (0,0.2);
            \draw (0.2,-0.5) to [out=135,in=down] (-0.15,-0.2) to[out=up,in=-135] (0,0);
        \end{tikzpicture}
        = \mergemor
        \ ,\quad
        \begin{tikzpicture}[centerzero]
            \draw  (0,-0.4) -- (0,-0.2) to[out=45,in=down] (0.15,0) to[out=up,in=-45] (0,0.2) -- (0,0.4);
            \draw (0,-0.2) to[out=135,in=down] (-0.15,0) to[out=up,in=-135] (0,0.2);
        \end{tikzpicture}
        = \alpha\
        \begin{tikzpicture}[centerzero]
            \draw(0,-0.4) -- (0,0.4);
        \end{tikzpicture}
        \ ,\quad
        \bubble = \delta 1_\one,
        \quad
        \lollydrop = 0.
    \end{gather}
\end{defin}

\begin{rem} \label{whyalpha}
    We can rescale $\capmor$ by $\alpha$ and $\cupmor$ by $\alpha^{-1}$ to see that $\Tcat_{\alpha,\delta}$ is isomorphic to $\Tcat_{1,\delta}$.  However, it will be useful in the forthcoming applications to include the parameter $\alpha$ in our definition.  In particular, we will be most interested in the case where $\alpha = \frac{7}{3}$ and $\delta=26$.
\end{rem}

\begin{rem} \label{meow}
    The relations in \cref{Tdef} all have conceptual category-theoretic meanings:
    \begin{itemize}
        \item The relations \cref{venom} and the last equality in \cref{vortex} correspond to the fact that the crossing endows $\Tcat$ with the structure of a symmetric monoidal category.
        \item The first two equalities in \cref{vortex}, together with the first equality in \cref{chess} assert that the generating object $\go$ is symmetrically self-dual.  Then the fourth equality in \cref{vortex} implies that $\Tcat$ is strict pivotal.  (See the discussion after \cref{windy}.)
        \item The second equality in \cref{chess} can be viewed as stating that $\mergemor$ corresponds to a commutative binary operation on $\go$.
        \item The fourth relation in \cref{chess} states that $\go$ has categorical dimension $\delta$.  (See \cref{sec:fundamental} for further discussion of categorical dimension.)
        \item If we want $\go$ to be a simple object (more precisely, $\Tcat(\go,\go) = \kk 1_\go$), not isomorphic to $\one$, then the fifth equality in \cref{chess} corresponds to the fact that there are no nonzero morphisms $\one \to \go$, while the left-hand side of the third equality in \cref{chess} must be a scalar multiple (which we denote by $\alpha$) of the identity $1_\go$.
    \end{itemize}
\end{rem}

Invariance of morphisms under rectilinear isotopy follows immediately from the interchange law in a monoidal category.  The next proposition will imply full isotopy invariance.  We define
\begin{equation}
    \dotcross
    :=
    \begin{tikzpicture}[centerzero]
        \draw (-0.2,-0.4) -- (-0.2,0.4);
        \draw (-0.2,-0.2) -- (0.4,0.4);
        \draw (0.4,-0.4) -- (-0.2,0.2);
    \end{tikzpicture}
    \ .
\end{equation}

\begin{prop} \label{windy}
    The following relations hold in $\Tcat$:
    \begin{gather} \label{topsy}
        \begin{tikzpicture}[anchorbase]
            \draw (-0.4,-0.2) to[out=up,in=180] (-0.2,0.2) to[out=0,in=135] (0,0);
            \draw (0,0) -- (0,-0.2);
            \draw (0.3,0.3) -- (0,0);
        \end{tikzpicture}
        =
        \mergemor
        =
        \begin{tikzpicture}[anchorbase]
            \draw (0.4,-0.2) to[out=up,in=0] (0.2,0.2) to[out=180,in=45] (0,0);
            \draw (0,0) -- (0,-0.2);
            \draw (-0.3,0.3) -- (0,0);
        \end{tikzpicture}
        \ ,\quad
        \triform
        :=
        \begin{tikzpicture}[centerzero]
          \draw (-0.2,-0.2) to (0,0);
          \draw (0.2,-0.2) to (0,0);
          \draw (0,0) arc(0:180:0.2) -- (-0.4,-0.2);
        \end{tikzpicture}
        =
        \begin{tikzpicture}[centerzero]
          \draw (-0.2,-0.2) to (0,0);
          \draw (0.2,-0.2) to (0,0);
          \draw (0,0) arc(180:0:0.2) -- (0.4,-0.2);
        \end{tikzpicture}
        \ ,\quad
        \explode
        :=
        \begin{tikzpicture}[centerzero]
          \draw (-0.2,0.2) to (0,0);
          \draw (0.2,0.2) to (0,0);
          \draw (0,0) arc(360:180:0.2) to (-0.4,0.2);
        \end{tikzpicture}
        =
        \begin{tikzpicture}[centerzero]
          \draw (-0.2,0.2) to (0,0);
          \draw (0.2,0.2) to (0,0);
          \draw (0,0) arc(180:360:0.2) to (0.4,0.2);
        \end{tikzpicture}
        \ ,
        \\ \label{turvy}
        \begin{tikzpicture}[centerzero]
            \draw (-0.2,0.3) -- (-0.2,0.1) arc(180:360:0.2) -- (0.2,0.3);
            \draw (-0.3,-0.3) to[out=up,in=down] (0,0.3);
        \end{tikzpicture}
        =
        \begin{tikzpicture}[centerzero]
            \draw (-0.2,0.3) -- (-0.2,0.1) arc(180:360:0.2) -- (0.2,0.3);
            \draw (0.3,-0.3) to[out=up,in=down] (0,0.3);
        \end{tikzpicture}
        \ ,\quad
        \begin{tikzpicture}[anchorbase]
            \draw (-0.2,0.2) -- (0.2,-0.2);
            \draw (-0.4,0.2) to[out=down,in=225,looseness=2] (0,0) to[out=45,in=up,looseness=2] (0.4,-0.2);
        \end{tikzpicture}
        =
        \crossmor
        =
        \begin{tikzpicture}[anchorbase]
            \draw (0.2,0.2) -- (-0.2,-0.2);
            \draw (0.4,0.2) to[out=down,in=-45,looseness=2] (0,0) to[out=135,in=up,looseness=2] (-0.4,-0.2);
        \end{tikzpicture}
        \ ,\quad
        \begin{tikzpicture}[anchorbase]
            \draw (-0.2,0.2) -- (0.2,-0.2);
            \draw (-0.4,0.2) to[out=down,in=225,looseness=2] (0,0) to[out=45,in=up,looseness=2] (0.4,-0.2);
            \opendot{0,0};
        \end{tikzpicture}
        =
        \dotcross
        =
        \begin{tikzpicture}[anchorbase]
            \draw (0.2,0.2) -- (-0.2,-0.2);
            \draw (0.4,0.2) to[out=down,in=-45,looseness=2] (0,0) to[out=135,in=up,looseness=2] (-0.4,-0.2);
            \opendot{0,0};
        \end{tikzpicture}
        \ .
    \end{gather}
\end{prop}

\begin{proof}
    The first and second equalities in \cref{topsy} follow immediately from the first four equalities in \cref{vortex}.  Then, using the first and second equalities in \cref{topsy}, we have
    \[
        \begin{tikzpicture}[centerzero]
          \draw (-0.2,-0.2) to (0,0);
          \draw (0.2,-0.2) to (0,0);
          \draw (0,0) arc(0:180:0.2) -- (-0.4,-0.2);
        \end{tikzpicture}
        =
        \begin{tikzpicture}[centerzero]
            \draw (0,-0.2) to (0,0) to[out=45,in=up,looseness=2] (0.3,-0.2);
            \draw (0,0) to[out=135,in=up,looseness=2] (-0.3,-0.2);
        \end{tikzpicture}
        =
        \begin{tikzpicture}[centerzero]
          \draw (-0.2,-0.2) to (0,0);
          \draw (0.2,-0.2) to (0,0);
          \draw (0,0) arc(180:0:0.2) -- (0.4,-0.2);
        \end{tikzpicture}
        \ ,
    \]
    proving the fourth equality in \cref{topsy}.  The proof of the sixth equality in \cref{topsy} is analogous.

    To prove the first equality in \cref{turvy}, we use \cref{vortex} to compute
    \[
        \begin{tikzpicture}[centerzero]
            \draw (-0.2,0.3) -- (-0.2,0.1) arc(180:360:0.2) -- (0.2,0.3);
            \draw (-0.3,-0.3) to[out=up,in=down] (0,0.3);
        \end{tikzpicture}
        =
        \begin{tikzpicture}[anchorbase]
            \draw (-1,0.5) -- (-1,0.2) arc(180:360:0.2) arc(180:0:0.2) arc(180:360:0.2) -- (0.2,0.5);
            \draw (-0.3,-0.3) \braidup (0,0.5);
        \end{tikzpicture}
        =
        \begin{tikzpicture}[anchorbase]
            \draw (-1,0.5) -- (-1,0.2) arc(180:360:0.2) arc(180:0:0.2) arc(180:360:0.2) -- (0.2,0.5);
            \draw (-0.5,-0.3) \braidup (-0.8,0.5);
        \end{tikzpicture}
        =
        \begin{tikzpicture}[centerzero]
            \draw (-0.2,0.3) -- (-0.2,0.1) arc(180:360:0.2) -- (0.2,0.3);
            \draw (0.3,-0.3) to[out=up,in=down] (0,0.3);
        \end{tikzpicture}
        \ .
    \]
    The second and third equalities in \cref{turvy} now follow from sliding the crossing over the cup or cap, and then using the first two equalities in \cref{vortex}.

    It remains to prove the fourth equality in \cref{turvy}, since the fifth equality then follows easily using the first two relations in \cref{vortex}.  Using the second and third relations in \cref{vortex} and the first two relations in \cref{topsy} to rotate trivalent vertices, we have
    \[
        \begin{tikzpicture}[anchorbase]
            \draw (-0.2,0.2) -- (0.2,-0.2);
            \draw (-0.4,0.2) to[out=down,in=225,looseness=2] (0,0) to[out=45,in=up,looseness=2] (0.4,-0.2);
            \opendot{0,0};
        \end{tikzpicture}
        =
        \begin{tikzpicture}[anchorbase]
            \draw (-0.2,0.4) -- (-0.2,0.2) -- (0,0) -- (0.2,0.2) -- (0.2,0.4);
            \draw (-0.2,0.2) -- (-0.4,0) -- (0,-0.4);
            \draw (0,0) -- (-0.4,-0.4);
        \end{tikzpicture}
        \ .
    \]
    Now, composing the fourth equality in \cref{venom} on the bottom with $\crossmor$ and on the top with
    $
        \begin{tikzpicture}[centerzero]
            \draw (-0.2,-0.2) -- (0,0.2);
            \draw (0,-0.2) -- (0.2,0.2);
            \draw (0.2,-0.2) -- (-0.2,0.2);
        \end{tikzpicture}
    $
    , and then using the first equality in \cref{venom}, we have
    \[
        \begin{tikzpicture}[anchorbase,scale=0.8]
            \draw (0.4,0.5) -- (-0.2,-0.3) -- (-0.4,-0.1) -- (0,0.5);
            \draw (-0.2,-0.3) -- (-0.2,-0.6);
            \draw (0.4,-0.6) -- (0.4,-0.1) -- (-0.4,0.5);
        \end{tikzpicture}
        =
        \begin{tikzpicture}[anchorbase]
            \draw (0.4,0.4) -- (0.2,0.2) -- (0.2,0) -- (-0.2,-0.4);
            \draw (0,0.4) -- (0.2,0.2);
            \draw (-0.2,0.4) -- (-0.2,0) -- (0.2,-0.4);
        \end{tikzpicture}
        \ .
    \]
    Using this and the second equality in \cref{chess}, we have
    \[
        \begin{tikzpicture}[anchorbase]
            \draw (-0.2,0.4) -- (-0.2,0.2) -- (0,0) -- (0.2,0.2) -- (0.2,0.4);
            \draw (-0.2,0.2) -- (-0.4,0) -- (0,-0.4);
            \draw (0,0) -- (-0.4,-0.4);
        \end{tikzpicture}
        =
        \begin{tikzpicture}[centerzero]
            \draw (-0.2,-0.4) -- (-0.2,0.4);
            \draw (-0.2,-0.2) -- (0.4,0.4);
            \draw (0.4,-0.4) -- (-0.2,0.2);
        \end{tikzpicture}
        =
        \dotcross\ ,
    \]
    completing the verification of the fourth equality in \cref{turvy}.
\end{proof}

It follows from \cref{vortex,topsy,turvy} that the cups and caps equip $\Tcat$ with the structure of a \emph{strict pivotal} category.  Intuitively, this means that morphisms are invariant under ambient isotopy fixing the boundary.  Thus, for example, it makes sense to allow horizontal strands in diagrams:
\begin{equation}
    \Hmor
    :=
    \begin{tikzpicture}[anchorbase]
        \draw (-0.4,-0.4) -- (-0.4,0) -- (-0.2,0.2) -- (0.2,-0.2) -- (0.4,0) -- (0.4,0.4);
        \draw (-0.2,0.2) -- (-0.2,0.4);
        \draw (0.2,-0.2) -- (0.2,-0.4);
    \end{tikzpicture}
    =
    \begin{tikzpicture}[anchorbase]
        \draw (0.4,-0.4) -- (0.4,0) -- (0.2,0.2) -- (-0.2,-0.2) -- (-0.4,0) -- (-0.4,0.4);
        \draw (0.2,0.2) -- (0.2,0.4);
        \draw (-0.2,-0.2) -- (-0.2,-0.4);
    \end{tikzpicture}
    \ .
\end{equation}
In addition, since the object $\go$ is self-dual, the cups and caps yield natural isomorphisms
\begin{equation} \label{twirl}
    \Tcat(\go^{\otimes m}, \go^{\otimes n})
    \cong \Tcat(\go^{\otimes (m+n)},\one),\qquad
    n,m \in \N.
\end{equation}

\begin{defin} \label{Fdef}
    Fix $\alpha \in \kk^\times$ and $\delta \in \kk$, with $\delta \ne -2$.  Let $\Fcat = \Fcat_{\alpha,\delta}$ be the strict monoidal category obtained from $\Tcat_{\alpha,\delta}$ by imposing the following three additional relations:
    \begin{gather} \label{magic}
        \Hmor + \Imor + \dotcross
        = \frac{2\alpha}{\delta+2}
        \left(\, \jail + \hourglass + \crossmor\, \right),
        \\ \label{sqburst}
        \sqmor
        =
        \frac{\alpha^2 (\delta + 14)}{2(\delta+2)^2} \left(\, \jail + \hourglass \, \right)
        + \frac{\alpha (\delta-6)}{2(\delta+2)} \left(\, \Hmor + \Imor\, \right)
        + \frac{3\alpha^2 (2-\delta)}{2(\delta+2)^2} \, \crossmor\ ,
        \\ \label{pentburst}
        \begin{aligned}
            \pentmor &=
            \frac{\alpha(10-\delta)}{4(\delta+2)}
            \left(
                \begin{tikzpicture}[anchorbase]
                    \draw (-0.2,0) -- (0,0.25) -- (0.2,0);
                    \draw (0,0.25) -- (0,0.4);
                    \draw (-0.2,-0.25) -- (-0.2,0) -- (-0.3,0.4);
                    \draw (0.2,-0.25) -- (0.2,0) -- (0.3,0.4);
                \end{tikzpicture}
                +
                \begin{tikzpicture}[anchorbase]
                    \draw (-0.3,0.3) -- (0,0) -- (0.3,0.3);
                    \draw (0,0.3) -- (-0.15,0.15);
                    \draw (0,0) -- (0,-0.15) -- (-0.15,-0.3);
                    \draw (0,-0.15) -- (0.15,-0.3);
                \end{tikzpicture}
                +
                \begin{tikzpicture}[anchorbase]
                    \draw (0.3,0.3) -- (0,0) -- (-0.3,0.3);
                    \draw (0,0.3) -- (0.15,0.15);
                    \draw (0,0) -- (0,-0.15) -- (0.15,-0.3);
                    \draw (0,-0.15) -- (-0.15,-0.3);
                \end{tikzpicture}
                +
                \begin{tikzpicture}[anchorbase]
                    \draw (-0.3,-0.3) -- (0.3,0.3);
                    \draw (-0.3,0.3) -- (-0.15,-0.15);
                    \draw (0,0.3) -- (0.15,0.15);
                    \draw (0,0) -- (0.3,-0.3);
                \end{tikzpicture}
                +
                \begin{tikzpicture}[anchorbase]
                    \draw (0.3,-0.3) -- (-0.3,0.3);
                    \draw (0.3,0.3) -- (0.15,-0.15);
                    \draw (0,0.3) -- (-0.15,0.15);
                    \draw (0,0) -- (-0.3,-0.3);
                \end{tikzpicture}
            \right)
            \\
            &\qquad - \frac{\alpha^2 (\delta+30)}{8(\delta+2)^2}
            \left(
                \begin{tikzpicture}[centerzero]
                    \draw (-0.15,-0.3) -- (-0.15,-0.23) arc(180:0:0.15) -- (0.15,-0.3);
                    \draw (-0.3,0.3) -- (0,0.08) -- (0.3,0.3);
                    \draw (0,0.3) -- (0,0.08);
                \end{tikzpicture}
                +
                \begin{tikzpicture}[centerzero]
                    \draw (-0.2,-0.3) -- (-0.2,0.3);
                    \draw (0,0.3) -- (0.15,0) -- (0.3,0.3);
                    \draw (0.15,0) -- (0.15,-0.3);
                \end{tikzpicture}
                +
                \begin{tikzpicture}[centerzero]
                    \draw (0.2,-0.3) -- (0.2,0.3);
                    \draw (0,0.3) -- (-0.15,0) -- (-0.3,0.3);
                    \draw (-0.15,0) -- (-0.15,-0.3);
                \end{tikzpicture}
                +
                \begin{tikzpicture}[centerzero]
                    \draw (-0.3,0.3) -- (-0.3,0.23) arc(180:360:0.15) -- (0,0.3);
                    \draw (0.3,0.3) -- (0.15,0) -- (-0.2,-0.3);
                    \draw (0.2,-0.3) -- (0.15,0);
                \end{tikzpicture}
                +
                \begin{tikzpicture}[centerzero]
                    \draw (0.3,0.3) -- (0.3,0.23) arc(360:180:0.15) -- (0,0.3);
                    \draw (-0.3,0.3) -- (-0.15,0) -- (0.2,-0.3);
                    \draw (-0.2,-0.3) -- (-0.15,0);
                \end{tikzpicture}
            \right)
            \\
            &\qquad + \frac{3 \alpha^2 (\delta-2)}{8(\delta+2)^2}
            \left(
                \begin{tikzpicture}[centerzero]
                    \draw (0,0.3) -- (0,-0.15) -- (-0.15,-0.3);
                    \draw (0,-0.15) -- (0.15,-0.3);
                    \draw (-0.2,0.3) -- (-0.2,0.25) arc(180:360:0.2) -- (0.2,0.3);
                \end{tikzpicture}
                +
                \begin{tikzpicture}[centerzero]
                    \draw (0,0.3) to[out=-45,in=70] (0.15,-0.3);
                    \draw (-0.3,0.3) -- (0,0) -- (0.3,0.3);
                    \draw (0,0) -- (-0.15,-0.3);
                \end{tikzpicture}
                +
                \begin{tikzpicture}[centerzero]
                    \draw (0,0.3) to[out=225,in=110] (-0.15,-0.3);
                    \draw (0.3,0.3) -- (0,0) -- (-0.3,0.3);
                    \draw (0,0) -- (0.15,-0.3);
                \end{tikzpicture}
                +
                \begin{tikzpicture}[centerzero]
                    \draw (-0.3,0.3) -- (-0.15,0.15) -- (0,0.3);
                    \draw (-0.15,0.15) -- (0.15,-0.3);
                    \draw (0.3,0.3) -- (-0.15,-0.3);
                \end{tikzpicture}
                +
                \begin{tikzpicture}[centerzero]
                    \draw (0.3,0.3) -- (0.15,0.15) -- (0,0.3);
                    \draw (0.15,0.15) -- (-0.15,-0.3);
                    \draw (-0.3,0.3) -- (0.15,-0.3);
                \end{tikzpicture}
            \right).
        \end{aligned}
    \end{gather}
\end{defin}

\begin{rem}
    We will see in \cref{prestige} that \cref{magic} corresponds to the Cayley--Hamilton theorem for traceless $3 \times 3$ octonionic matrices (see \cref{boysenberry,mango}).  Note also that, if $\delta \ne 2$, then \cref{sqburst} allows one to write $\crossmor$ in terms of diagrams built from $\mergemor$, $\cupmor$, and $\capmor$.  Thus $\Fcat$ is a \emph{trivalent category} in the sense of \cite{MPS17}.
\end{rem}

Before proceeding further, let us motivate the assumption in \cref{Fdef} that $\delta \ne -2$.  In fact, we could multiply both sides of \cref{magic,sqburst,pentburst} by an appropriate power of $\delta+2$ to clear this factor from the denominators.  For instance, \cref{magic} then becomes
\begin{equation} \label{jordan}
    (\delta+2) \left( \Hmor + \Imor + \dotcross \right)
    = 2 \alpha
    \left(\, \jail + \hourglass + \crossmor\, \right).
\end{equation}
Then it makes sense to allow $\delta = -2$.  In this case, \cref{jordan} gives
$
    \crossmor
    \ = -\
    \begin{tikzpicture}[centerzero]
        \draw (-0.15,-0.3) -- (-0.15,0.3);
        \draw (0.15,-0.3) -- (0.15,0.3);
    \end{tikzpicture}
    \ -\
    \begin{tikzpicture}[centerzero]
        \draw (-0.15,-0.3) -- (-0.15,-0.25) arc(180:0:0.15) -- (0.15,-0.3);
        \draw (-0.15,0.3) -- (-0.15,0.25) arc(180:360:0.15) -- (0.15,0.3);
    \end{tikzpicture}
    \ .
$
Composing on the top with $\mergemor$ and using the second and fifth equalities in \cref{chess} then gives $\mergemor=0$.  (Here we use that the characteristic of the field is not $2$.)  But then the third equality in \cref{chess} implies that $1_\one = 0$, and hence the category $\Fcat$ collapses to the trivial category.  On the other hand, omitting \cref{chess} from the presentation would result in the \emph{Temperley--Lieb category}, which is the strict monoidal category generated by the object $\go$ and morphisms $\cupmor$, $\capmor$, subject to the first two equalities in \cref{vortex} and the fourth equality in \cref{chess}.

\section{Dimension restrictions}

In this section we show that, with some mild restrictions on $\delta$, any quotient of the category $\Tcat$ with certain conditions on the dimensions of the morphism spaces $\go^{\otimes 2} \to \go^{\otimes 2}$ and $\go^{\otimes 2} \to \go^{\otimes 3}$ satisfies the additional relations \cref{magic,sqburst,pentburst}, and hence is also a quotient of $\Fcat$.  Later, in \cref{sec:functor}, we will use this result to show that the functor $\Tcat \to \fg\md$ defined in \cref{magneto} factors through $\Fcat$ (\cref{baja}).

Recall that an \emph{ideal} in a $\kk$-linear category $\cC$ is a collection $\cI$ of vector subspaces $\cI(X,Y)$ of $\cC(X,Y)$ for all $X,Y \in \Ob(\cC)$ such that
\[
    \cC(Y,Z) \circ \cI(X,Y) \subseteq \cI(X,Z)
    \quad \text{and} \quad
    \cI(X,Y) \circ \cC(Z,X) \subseteq \cI(Z,Y)
\]
for all $X,Y,Z \in \Ob(\cC)$.  If, in addition, $\cC$ is a monoidal category, then we say $\cI$ is a \emph{tensor ideal} if it is an ideal and
\[
    1_Z \otimes \cI(X,Y) \subseteq \cI(Z \otimes X, Z \otimes Y)
    \quad \text{and} \quad
    \cI(X,Y) \otimes 1_Z \subseteq \cI(X \otimes Z, Y \otimes Z)
\]
for all $X,Y,Z \in \Ob(\cC)$.  If $\cI$ is a tensor ideal, it follows that $f \otimes g$ and $g \otimes f$ belong to $\cI$ for arbitrary morphisms $f$ in $\cI$ and $g$ in $\cC$.  If $\cI$ is a tensor ideal of $\cC$, then the \emph{quotient category} $\cC/\cI$ is the category with
\[
    \Ob (\cC/\cI) = \Ob(\cC),\quad
    (\cC/\cI)(X,Y) = \cC(X,Y) / \cI(X,Y).
\]
The composition and tensor product in $\cC/\cI$ are induced by those in $\cC$.

\begin{theo} \label{demayo}
    Assume $\delta \notin \{-2,2,6,10\}$.  If $\cI$ is a tensor ideal of $\Tcat$ such that
    \begin{equation} \label{cinco}
        \dim \left( (\Tcat/\cI)(\go^{\otimes 2}, \go^{\otimes 2}) \right) = 5
        \quad \text{and} \quad
        \dim \left( (\Tcat/\cI)(\go^{\otimes 3}, \go^{\otimes 2}) \right) \le 15
    \end{equation}
    then relations \cref{magic,sqburst,pentburst} hold in $\Tcat/\cI$.
\end{theo}

The proof of \cref{demayo} will occupy the remainder of this section.  We break the proof into a series of smaller results.
\begin{center}
    \textit{We assume for the remainder of this section that $\delta \ne -2$.}
\end{center}

For $m,n \ge 1$, consider the linear operators
\begin{align*}
    \Rot \colon \Tcat(\go^{\otimes m}, \go^{\otimes n}) &\to \Tcat(\go^{\otimes m}, \go^{\otimes n}),
    &
    \Rot
    \left(
        \begin{tikzpicture}[anchorbase]
            \draw[line width=2] (-0.1,-0.4) -- (-0.1,0);
            \draw (0.1,-0.4) -- (0.1,0);
            \draw (-0.1,0) -- (-0.1,0.4);
            \draw[line width=2] (0.1,0) -- (0.1,0.4);
            \filldraw[fill=white,draw=black] (-0.25,0.2) rectangle (0.25,-0.2);
            \node at (0,0) {$\scriptstyle{f}$};
        \end{tikzpicture}
    \right)
    &=
    \begin{tikzpicture}[anchorbase]
        \draw (-0.25,0.2) rectangle (0.25,-0.2);
        \node at (0,0) {$\scriptstyle{f}$};
        \draw (-0.4,-0.4) -- (-0.4,0.2) arc (180:0:0.15);
        \draw (0.4,0.4) -- (0.4,-0.2) arc(360:180:0.15);
        \draw[line width=2] (-0.1,-0.4) -- (-0.1,-0.2);
        \draw[line width=2] (0.1,0.4) -- (0.1,0.2);
    \end{tikzpicture}
    \ ,
    \\
    \Switch \colon \Tcat(\go^{\otimes 2}, \go^{\otimes n}) &\to \Tcat(\go^{\otimes 2}, \go^{\otimes n}),
    &
    \Switch
    \left(
        \begin{tikzpicture}[anchorbase]
            \draw (-0.1,-0.4) -- (-0.1,0);
            \draw (0.1,-0.4) -- (0.1,0);
            \draw[line width=2] (0,0.2) -- (0,0.4);
            \filldraw[fill=white,draw=black] (-0.25,0.2) rectangle (0.25,-0.2);
            \node at (0,0) {$\scriptstyle{f}$};
        \end{tikzpicture}
    \right)
    &=
    \begin{tikzpicture}[anchorbase]
        \draw (-0.1,-0.5) -- (0.1,-0.2);
        \draw (0.1,-0.5) -- (-0.1,-0.2);
        \draw[line width=2] (0,0.2) -- (0,0.4);
        \filldraw[fill=white,draw=black] (-0.25,0.2) rectangle (0.25,-0.2);
        \node at (0,0) {$\scriptstyle{f}$};
    \end{tikzpicture}
    \ ,
\end{align*}
where the bottom and top thick strands in the definition of $\Rot$ represent $1_\go^{\otimes (m-1)}$ and $1_\go^{\otimes (n-1)}$, respectively, and the thick strand in the definition of $\Switch$ represents $1_\go^{\otimes n}$.  In other words,
\[
    \Rot(f) = \left( \capmor \otimes 1_\go^{\otimes n} \right) \circ (1_\go \otimes f \otimes 1_\go) \circ (1_\go^{\otimes m} \otimes \cupmor),\quad
    \Switch(f) = f \circ \crossmor.
\]
Any tensor ideal of $\Tcat$ or $\Fcat$ is invariant under $\Rot$ and $\Switch$.  We have
\begin{equation} \label{rotary}
    \begin{aligned}
        \Rot \left(\, \jail\, \right) &= \hourglass\, ,&
        \Rot \left(\, \hourglass\, \right) &= \jail\, ,&
        \Rot \left(\, \crossmor\, \right) &= \crossmor\, ,
        \\
        \Rot \left(\, \Hmor\, \right) &= \Imor\, ,&
        \Rot \left(\, \Imor\, \right) &= \Hmor\, ,&
        \Rot \left(\, \dotcross\, \right) &= \dotcross\, ,
    \end{aligned}
\end{equation}
and
\begin{equation} \label{flick}
    \begin{aligned}
        \Switch \left(\, \jail\, \right) &= \crossmor\, ,&
        \Switch \left(\, \crossmor\, \right) &= \jail\, ,&
        \Switch \left(\, \hourglass\, \right) &= \hourglass\, ,
        \\
        \Switch \left(\, \Hmor\, \right) &= \dotcross\, ,&
        \Switch \left(\, \dotcross\, \right) &= \Hmor\, ,&
        \Switch \left(\, \Imor\, \right) &= \Imor\, .
    \end{aligned}
\end{equation}

\begin{prop} \label{SUP}
    If $\cI$ is a tensor ideal of $\Tcat$ such that
    \begin{equation} \label{funf}
        \dim \left( (\Tcat/\cI)(\go^{\otimes 2}, \go^{\otimes 2}) \right) = 5,
    \end{equation}
    then \cref{magic} holds in $\Tcat/\cI$, and the morphisms
    \begin{equation} \label{bigfive}
        \jail\, ,\quad \hourglass\, ,\quad \crossmor\, ,\quad \Hmor\, ,\quad \Imor
    \end{equation}
    give a basis for $(\Tcat/\cI)(\go^{\otimes 2}, \go^{\otimes 2})$.
\end{prop}

\begin{proof}
    Suppose $\cI$ is a tensor ideal of $\Tcat$ satisfying \cref{funf}, and let $\cC = \Tcat/\cI$.  Then, in $\cC$, there must be a linear dependence relation involving the six morphisms
    \begin{equation} \label{bigsix}
        \jail\, ,\quad \hourglass\, ,\quad \crossmor\, ,\quad \Hmor\, ,\quad \Imor\, ,\quad \dotcross\, .
    \end{equation}
    Let us for a moment view the diagrams in \cref{bigsix} as graphs (with $\dotcross$ being a 4-valent vertex) embedded in the plane.  The operators $\Rot$ and $\Switch$ act on this set as in \cref{rotary,flick}.  It follows that $\Rot$ and $\Switch$ generate an action of the symmetric group $\fS_3$ on the $6$-dimensional space $U$ spanned by $\cref{bigsix}$ (viewed only as embedded planar graphs).  The space $U$ decomposes as a direct sum
    \[
        U = U_1 \oplus U_2 \oplus U_3 \oplus U_4,
    \]
    where
    \[
        U_1 = \Span_\kk \left\{ \jail\, +\, \hourglass\, +\, \crossmor \right\}
        \quad \text{and} \quad
        U_2 = \Span_\kk \left\{ \Hmor\, +\, \Imor\, +\,  \dotcross \right\}
    \]
    are copies of the trivial $\fS_3$-module,
    \[
        U_3 = \Span_\kk \left\{ \jail\, +\, \hourglass\, -2\, \crossmor\, ,\ \jail\, -\, \hourglass \right\}
        \quad \text{and} \quad
        U_4 = \Span_\kk \left\{ \Hmor\, +\, \Imor\, -2\, \dotcross\, ,\ \Hmor\, -\, \Imor \right\}
    \]
    are copies of the unique simple $\fS_3$-module of dimension $2$ (that is, the Specht module corresponding to the partition $(2,1)$), and there is an isomorphism $U_3 \xrightarrow{\cong} U_4$ sending the given basis of $U_3$ to the given basis of $U_4$.  If some linear combination $u$ of the elements \cref{bigsix} is zero in $\Tcat/\cI$, then all elements of the $\fS_3$-submodule generated by $u$ are also zero in $\Tcat/\cI$.

    First consider the case where $\cI(\go^{\otimes 2}, \go^{\otimes 2}) \cap (U_3 \oplus U_4) \ne \{0\}$.  By the above discussion, $\cI$ then contains at least the span of the vectors
    \[
        \lambda \left( \Hmor - \Imor \right)
        + \mu \left( \jail\, -\, \hourglass\, \right)
        \quad \text{and} \quad
        \lambda \left( \Hmor\, +\, \Imor\, -2\, \dotcross \right)
        + \mu \left(\, \jail\, +\, \hourglass\, -2\, \crossmor \right)
    \]
    for some $\lambda,\mu \in \kk$, not both zero.  Now, if $\lambda = 0$, then we have $\jail\, =\, \hourglass$.  Composing on the top with $\mergemor$ and using the fourth relation in \cref{chess}, we then get $\mergemor=0$.  But then $\Tcat/\cI$ is trivial, as noted at the end of \cref{sec:defin}.  This contradicts our hypothesis \cref{funf}.

    Thus, we may suppose $\lambda \ne 0$, in which case relations of the form
    \begin{equation} \label{croatia}
        \Hmor\, -\, \Imor\, =\, \mu \left(\, \jail\, -\, \hourglass\, \right)
        \quad \text{and} \quad
        \Hmor\, +\, \Imor\, -2\, \dotcross\, =\, \mu \left(\, \jail\, +\, \hourglass\, -2\, \crossmor\, \right)
    \end{equation}
    hold in $\Tcat/\cI$.  From the first equation in \cref{croatia}, we have
    \[
        \begin{tikzpicture}[centerzero]
            \draw (-0.3,-0.2) -- (-0.1,0);
            \draw (0.3,-0.2) -- (0.1,0);
            \draw[thick,densely dotted] (-0.3,0.2) -- (-0.1,0) -- (0.1,0) -- (0.3,0.2);
        \end{tikzpicture}
        \, =\, \Imor\, + \mu \left(\, \jail\, -\, \hourglass\, \right),
    \]
    which allows us to reduce the length of any cycle of length at least three, or break open the cycle.  (The part of a cycle that would be replaced is indicated by dotted lines.)  Thus, $(\Tcat/\cI)(\go^{\otimes 2}, \go^{\otimes 2})$ is spanned by acyclic diagrams.  The second relation in \cref{croatia} then allows us to eliminate $\dotcross$.  This implies that $\jail\,$, $\hourglass\,$, $\Hmor$, and $\Imor$ span $(\Tcat/\cI)(\go^{\otimes 2}, \go^{\otimes 2})$, since these are the only planar acyclic trivalent graphs connecting four endpoints.  This contradicts our hypothesis \cref{funf}.

    We now know that $\cI(\go^{\otimes 2}, \go^{\otimes 2}) \subseteq U_1 \oplus U_2$.  So $\cI$ contains at least the span of the vectors
    \[
        \lambda \left(\, \jail\, +\, \hourglass\, +\, \crossmor\, \right)
        + \mu \left( \Hmor + \Imor + \dotcross\, \right)
    \]
    for some $\lambda,\mu \in \kk$, not both zero.  If $\mu = 0$, then we have
    \[
        \jail\, +\, \hourglass\, +\, \crossmor\, = 0
        \quad \text{in } \Tcat/\cI.
    \]
    Composing on the bottom with $\cupmor$ yields $(\delta+2)\, \cupmor = 0$.  If $\cupmor=0$, then $1_\go = 0$ by the first relation in \cref{vortex}, and so $\Tcat/\cI$ is the trivial category.  On the other hand, if $\delta=-2$ then $\Tcat/\cI$ is again trivial, as noted at the end of \cref{sec:defin}.  This contradicts our hypothesis \cref{funf}.

    We may thus assume $\mu \ne 0$.  Hence a relation of the form
    \[
        \Hmor + \Imor + \dotcross
        = \lambda
        \left(\, \jail + \hourglass + \crossmor\, \right)
    \]
    holds in $\Tcat/\cI$.  Now, composing on the bottom with $\cupmor$ and using \cref{chess}, we have
    \[
        2 \alpha\, \cupmor\, = \lambda (\delta+2)\, \cupmor,
    \]
    which implies that $\lambda = \frac{2\alpha}{\delta+2}$.  (As explained above, we cannot have $\cupmor\, =0$.)  Therefore, \cref{magic} holds in $\Tcat/\cI$.

    It remains to prove that the morphisms \cref{bigfive} give a basis for $(\Tcat/\cI)(\go^{\otimes 2}, \go^{\otimes 2})$.  In light of our assumption \cref{funf}, it suffices to show that the morphisms \cref{bigfive} are linearly independent.  In fact, this already follows from the above discussion.  As we saw above, no nonzero element of $U_3 \oplus U_4$ can be zero in $\Tcat/\cI$, and the space $U_1 \oplus U_2$ has dimension one in $\Tcat/\cI$.  Therefore, in $\Tcat/\cI$, the span of the morphisms \cref{bigfive} is $5$-dimensional, as required.
\end{proof}

Note that, at this point, we have not yet proved that a tensor ideal $\cI$ satisfying \cref{funf} exists.  However, it will follow from \cref{SW} that such an ideal does exist for $\alpha=7/3$ and $\delta=26$.

\begin{rem} \label{zagreb}
    Composing the relations in \cref{croatia} on the bottom with $\cupmor$ shows that $(\delta-1)\lambda = \alpha$.  (We assume here that $\cupmor\, \ne 0$, since, as we saw in the proof of \cref{SUP}, $\cupmor=0$ would imply that $\Tcat/\cI$ is trivial.)  Thus, there is a category $\mathcal{D}$ obtained from $\Tcat$ by imposing the additional relations
    \begin{equation} \label{bosnia}
        \Hmor\, -\, \Imor\, =\, \frac{\alpha}{\delta-1} \left(\, \jail\, -\, \hourglass\, \right)
        \quad \text{and} \quad
        \Hmor\, +\, \Imor\, -2\, \dotcross\, =\, \frac{\alpha}{\delta-1} \left(\, \jail\, +\, \hourglass\, -2\, \crossmor\, \right).
    \end{equation}
    (We have assumed here that $\delta \ne 1$, since $\delta=1$ leads to a rather uninteresting category where $\go^{\otimes 2} \cong \one$.)

    Other than the trivial category and the category with $\go^{\otimes 2} \cong \one$, the above argument shows that there are precisely \emph{two} quotients of $\Tcat$ whose morphisms spaces $\go^{\otimes 2} \to \go^{\otimes 2}$ have dimension less than or equal to $5$.  Whereas the goal of the current paper is to examine the category $\Fcat$ and relate it to the representation theory of the Lie algebra of type $F_4$, the authors do not know of the significance of the category $\mathcal{D}$ in representation theory.  We feel this category merits further investigation.
\end{rem}

\begin{lem} \label{spinner}
    If $\cI$ is a tensor ideal of $\Tcat$ such that \cref{magic} holds in $\Tcat/\cI$, then the relation
    \begin{equation} \label{triangle}
        \trimor = \frac{\alpha(2-\delta)}{2(\delta+2)}\, \mergemor
    \end{equation}
    holds in $\Tcat/\cI$.  In particular, \cref{triangle} holds in $\Fcat$.
\end{lem}

\begin{proof}
    Relation \cref{triangle} follows by composing \cref{magic} on the top with $\mergemor$, then using \cref{chess,turvy}.
\end{proof}

The first two relations in \cref{venom} imply that we have an algebra homomorphism
\begin{equation}
    \kk \fS_n \to \Tcat(\go^{\otimes n}, \go^{\otimes n}),
\end{equation}
sending the simple transposition $s_i \in \fS_n$ to the crossing of the $i$-th and $(i+1)$-st strands.  We will denote the image of the complete symmetrizers and antisymmetrizers under this homomorphism by white and black rectangles, respectively:
\begin{equation} \label{boxes}
    \begin{tikzpicture}[centerzero]
        \draw (-0.5,-0.6) -- (-0.5,0.6);
        \draw (0.5,-0.6) -- (0.5,0.6);
        \node at (0,0.4) {$\cdots$};
        \node at (0,-0.4) {$\cdots$};
        \symbox{-0.7,-0.15}{0.7,0.15};
    \end{tikzpicture}
    = \frac{1}{n!} \sum_{\sigma \in \fS_n}
    \begin{tikzpicture}[centerzero]
        \draw (-0.5,-0.6) -- (-0.5,0.6);
        \draw (0.5,-0.6) -- (0.5,0.6);
        \node at (0,0.4) {$\cdots$};
        \node at (0,-0.4) {$\cdots$};
        \filldraw[rounded corners, fill=white, draw=black] (-0.7,-0.15) rectangle (0.7,0.15);
        \node at (0,0) {$\sigma$};
    \end{tikzpicture}
    \ ,\qquad
    \begin{tikzpicture}[centerzero]
        \draw (-0.5,-0.6) -- (-0.5,0.6);
        \draw (0.5,-0.6) -- (0.5,0.6);
        \node at (0,0.4) {$\cdots$};
        \node at (0,-0.4) {$\cdots$};
        \antbox{-0.7,-0.15}{0.7,0.15};
    \end{tikzpicture}
    = \frac{1}{n!} \sum_{\sigma \in \fS_n} (-1)^{\ell(\sigma)}
    \begin{tikzpicture}[centerzero]
        \draw (-0.5,-0.6) -- (-0.5,0.6);
        \draw (0.5,-0.6) -- (0.5,0.6);
        \node at (0,0.4) {$\cdots$};
        \node at (0,-0.4) {$\cdots$};
        \filldraw[rounded corners, fill=white, draw=black] (-0.7,-0.15) rectangle (0.7,0.15);
        \node at (0,0) {$\sigma$};
    \end{tikzpicture}
    \ ,
\end{equation}
where the diagrams contain $n$ strings, $\fS_n$ is the symmetric group on $n$ letters, and $\ell(\sigma)$ is the length of the permutation $\sigma$ (i.e.\ the number of simple transpositions appearing in a reduced expression of $\sigma$).  It then follows from \cref{venom,chess} that
\begin{equation} \label{pomegranate}
    \begin{tikzpicture}[centerzero]
        \draw (-0.15,-0.35) -- (-0.15,0.2) arc(180:0:0.15) -- (0.15,-0.35);
        \symbox{-0.25,-0.1}{0.25,0.1};
    \end{tikzpicture}
    = \capmor\, ,\quad
    \begin{tikzpicture}[centerzero]
        \draw (-0.15,-0.35) -- (-0.15,0.2) arc(180:0:0.15) -- (0.15,-0.35);
        \antbox{-0.25,-0.1}{0.25,0.1};
    \end{tikzpicture}
    = 0,\quad
    \begin{tikzpicture}[centerzero]
        \draw (-0.15,-0.35) -- (-0.15,0);
        \draw (0.15,-0.35) -- (0.15,0);
        \symbox{-0.25,-0.1}{0.25,0.1};
        \draw (-0.15,0.1) -- (0.15,0.35);
        \draw (0.15,0.1) -- (-0.15,0.35);
    \end{tikzpicture}
    =
    \begin{tikzpicture}[centerzero]
        \draw (-0.15,-0.35) -- (-0.15,0.35);
        \draw (0.15,-0.35) -- (0.15,0.35);
        \symbox{-0.25,-0.1}{0.25,0.1};
    \end{tikzpicture}
    \, ,\quad
    \begin{tikzpicture}[centerzero]
        \draw (-0.15,-0.35) -- (-0.15,0);
        \draw (0.15,-0.35) -- (0.15,0);
        \antbox{-0.25,-0.1}{0.25,0.1};
        \draw (-0.15,0.1) -- (0.15,0.35);
        \draw (0.15,0.1) -- (-0.15,0.35);
    \end{tikzpicture}
    = -\,
    \begin{tikzpicture}[centerzero]
        \draw (-0.15,-0.35) -- (-0.15,0.35);
        \draw (0.15,-0.35) -- (0.15,0.35);
        \antbox{-0.25,-0.1}{0.25,0.1};
    \end{tikzpicture}
    \, ,\quad
    \begin{tikzpicture}[centerzero]
        \draw (-0.15,-0.35) -- (-0.15,0);
        \draw (0.15,-0.35) -- (0.15,0);
        \symbox{-0.25,-0.1}{0.25,0.1};
        \draw (-0.15,0.1) -- (0,0.25) -- (0,0.4);
        \draw (0.15,0.1) -- (0,0.25);
    \end{tikzpicture}
    = \mergemor\, ,\quad
    \begin{tikzpicture}[centerzero]
        \draw (-0.15,-0.35) -- (-0.15,0);
        \draw (0.15,-0.35) -- (0.15,0);
        \antbox{-0.25,-0.1}{0.25,0.1};
        \draw (-0.15,0.1) -- (0,0.25) -- (0,0.4);
        \draw (0.15,0.1) -- (0,0.25);
    \end{tikzpicture}
    = 0.
\end{equation}
It also follows from \cref{turvy} that
\begin{equation} \label{ladderslip}
    \begin{tikzpicture}[anchorbase]
        \draw (-0.2,-0.4) -- (-0.2,0.6);
        \draw (0.2,-0.4) -- (0.2,0.6);
        \draw (-0.2,0.35) -- (0.2,0.35);
        \symbox{-0.3,-0.1}{0.3,0.1};
    \end{tikzpicture}
    =
    \begin{tikzpicture}[anchorbase]
        \draw (-0.2,0.4) -- (-0.2,-0.6);
        \draw (0.2,0.4) -- (0.2,-0.6);
        \draw (-0.2,-0.35) -- (0.2,-0.35);
        \symbox{-0.3,-0.1}{0.3,0.1};
    \end{tikzpicture}
    \ ,\qquad
    \begin{tikzpicture}[anchorbase]
        \draw (-0.2,-0.4) -- (-0.2,0.6);
        \draw (0.2,-0.4) -- (0.2,0.6);
        \draw (-0.2,0.35) -- (0.2,0.35);
        \antbox{-0.3,-0.1}{0.3,0.1};
    \end{tikzpicture}
    =
    \begin{tikzpicture}[anchorbase]
        \draw (-0.2,0.4) -- (-0.2,-0.6);
        \draw (0.2,0.4) -- (0.2,-0.6);
        \draw (-0.2,-0.35) -- (0.2,-0.35);
        \antbox{-0.3,-0.1}{0.3,0.1};
    \end{tikzpicture}
    \ .
\end{equation}

\begin{lem} \label{sqexplode}
    If $\cI$ is a tensor ideal of $\Tcat$ satisfying \cref{funf}, then the relation \cref{sqburst} holds in $\Tcat/\cI$.
\end{lem}

\begin{proof}
    By \cref{SUP,spinner}, the morphisms \cref{bigfive} give a basis for $(\Tcat/\cI)(\go^{\otimes 2},\go^{\otimes 2})$ and \cref{triangle} holds.  Since $\sqmor$ is invariant under $\Rot$, we must have a relation in $\Tcat/\cI$ of the form
    \begin{equation} \label{sqbreak1}
        \sqmor =
        \beta_1 \left(\, \jail + \hourglass \, \right)
        + \beta_2 \left(\, \Hmor + \Imor\, \right)
        + \beta_3\, \crossmor\ ,\qquad \beta_1,\beta_2,\beta_3 \in \kk.
    \end{equation}
    Attaching a symmetrizer to the bottom of the diagrams in \cref{magic} and using \cref{pomegranate} gives
    \begin{equation} \label{sqbreak2}
        2\
        \begin{tikzpicture}[anchorbase]
            \draw (-0.2,-0.4) -- (-0.2,0.6);
            \draw (0.2,-0.4) -- (0.2,0.6);
            \draw (-0.2,0.35) -- (0.2,0.35);
            \symbox{-0.3,-0.1}{0.3,0.1};
        \end{tikzpicture}
        + \Imor
        =
        \frac{2 \alpha}{\delta+2}
        \left(\,
            \hourglass + 2\
            \begin{tikzpicture}[centerzero]
                \draw (-0.15,-0.35) -- (-0.15,0.35);
                \draw (0.15,-0.35) -- (0.15,0.35);
                \symbox{-0.25,-0.1}{0.25,0.1};
            \end{tikzpicture}\,
        \right).
    \end{equation}
    Attaching a $\Hmor$ to the top of the diagrams in \cref{sqbreak2} and using \cref{chess,triangle} gives
    \begin{equation} \label{sqbreak3}
        \begin{tikzpicture}[anchorbase]
            \draw (-0.2,-0.4) -- (-0.2,0.6);
            \draw (0.2,-0.4) -- (0.2,0.6);
            \draw (-0.2,0.2) -- (0.2,0.2);
            \draw (-0.2,0.4) -- (0.2,0.4);
            \symbox{-0.3,-0.2}{0.3,0};
        \end{tikzpicture}
        =
        \frac{\alpha}{\delta+2}
        \left(\,
            \alpha\, \hourglass + 2\
            \begin{tikzpicture}[anchorbase]
                \draw (-0.2,-0.4) -- (-0.2,0.6);
                \draw (0.2,-0.4) -- (0.2,0.6);
                \draw (-0.2,0.35) -- (0.2,0.35);
                \symbox{-0.3,-0.1}{0.3,0.1};
            \end{tikzpicture}\,
        \right)
        + \frac{\alpha(\delta-2)}{4(\delta+2)}\, \Imor
        \overset{\cref{sqbreak2}}{=}
        \frac{4 \alpha^2}{(\delta+2)^2}\,
        \begin{tikzpicture}[centerzero]
            \draw (-0.15,-0.35) -- (-0.15,0.35);
            \draw (0.15,-0.35) -- (0.15,0.35);
            \symbox{-0.25,-0.1}{0.25,0.1};
        \end{tikzpicture}
        \, + \frac{(\delta+4) \alpha^2}{(\delta+2)^2}\, \hourglass\,
        + \frac{(\delta-6)\alpha}{4(\delta+2)}\, \Imor.
    \end{equation}
    On the other hand, attaching a symmetrizer to the bottom of \cref{sqbreak1} gives
    \begin{equation} \label{sqbreak4}
        \begin{aligned}
            \begin{tikzpicture}[anchorbase]
                \draw (-0.2,-0.4) -- (-0.2,0.6);
                \draw (0.2,-0.4) -- (0.2,0.6);
                \draw (-0.2,0.2) -- (0.2,0.2);
                \draw (-0.2,0.4) -- (0.2,0.4);
                \symbox{-0.3,-0.2}{0.3,0};
            \end{tikzpicture}
            &=
            \beta_1
            \left(\,
                \begin{tikzpicture}[centerzero]
                    \draw (-0.15,-0.35) -- (-0.15,0.35);
                    \draw (0.15,-0.35) -- (0.15,0.35);
                    \symbox{-0.25,-0.1}{0.25,0.1};
                \end{tikzpicture}\,
                + \hourglass
            \, \right)
            + \beta_2
            \left(\,
                \begin{tikzpicture}[anchorbase]
                    \draw (-0.2,-0.4) -- (-0.2,0.6);
                    \draw (0.2,-0.4) -- (0.2,0.6);
                    \draw (-0.2,0.35) -- (0.2,0.35);
                    \symbox{-0.3,-0.1}{0.3,0.1};
                \end{tikzpicture}\,
                + \Imor
            \, \right)
            + \beta_3\,
            \begin{tikzpicture}[centerzero]
                \draw (-0.15,-0.35) -- (-0.15,0.35);
                \draw (0.15,-0.35) -- (0.15,0.35);
                \symbox{-0.25,-0.1}{0.25,0.1};
            \end{tikzpicture}
            \\
            &\overset{\mathclap{\cref{sqbreak2}}}{=}\
            \left(
                \beta_1 + \beta_3 + \frac{2 \alpha \beta_2}{\delta+2}
            \right)\,
            \begin{tikzpicture}[centerzero]
                \draw (-0.15,-0.35) -- (-0.15,0.35);
                \draw (0.15,-0.35) -- (0.15,0.35);
                \symbox{-0.25,-0.1}{0.25,0.1};
            \end{tikzpicture}
            \, + \left( \beta_1 + \frac{\alpha \beta_2}{\delta+2} \right) \hourglass
            \, + \frac{\beta_2}{2}\, \Imor.
        \end{aligned}
    \end{equation}
    Since the morphisms \cref{bigfive} are linearly independent, comparing \cref{sqbreak3,sqbreak4} gives
    \[
        \frac{4 \alpha^2}{(\delta+2)^2} = \beta_1 + \beta_3 + \frac{2 \alpha}{\delta+2} \beta_2,\qquad
        \frac{(\delta+4)\alpha^2}{(\delta+2)^2} = \beta_1 + \frac{\alpha}{\delta+2} \beta_2,\qquad
        \frac{(\delta-6)\alpha}{4(\delta+2)} = \frac{1}{2}\beta_2.
    \]
    Solving this linear system for $\beta_1$, $\beta_2$, and $\beta_3$ gives the coefficients of relation \cref{sqburst}.
\end{proof}

\begin{lem}
    We have
    \begin{equation} \label{3spike}
        \begin{tikzpicture}[anchorbase]
            \draw (-0.2,0) -- (0,0.25) -- (0.2,0);
            \draw (0,0.25) -- (0,0.4);
            \draw (0.2,-0.25) -- (-0.2,0) -- (-0.3,0.4);
            \draw (-0.2,-0.25) -- (0.2,0) -- (0.3,0.4);
        \end{tikzpicture}
        =\!
        \begin{tikzpicture}[anchorbase]
            \draw (-0.3,0.3) -- (0,0) -- (0.3,0.3);
            \draw (0,0.3) -- (-0.15,0.15);
            \draw (0,0) -- (0,-0.15) -- (-0.15,-0.3);
            \draw (0,-0.15) -- (0.15,-0.3);
        \end{tikzpicture}
        \!\! +\!\!
        \begin{tikzpicture}[anchorbase]
            \draw (0.3,0.3) -- (0,0) -- (-0.3,0.3);
            \draw (0,0.3) -- (0.15,0.15);
            \draw (0,0) -- (0,-0.15) -- (0.15,-0.3);
            \draw (0,-0.15) -- (-0.15,-0.3);
        \end{tikzpicture}
        \!\! -\!
        \begin{tikzpicture}[anchorbase]
            \draw (-0.2,0) -- (0,0.25) -- (0.2,0);
            \draw (0,0.25) -- (0,0.4);
            \draw (-0.2,-0.25) -- (-0.2,0) -- (-0.3,0.4);
            \draw (0.2,-0.25) -- (0.2,0) -- (0.3,0.4);
        \end{tikzpicture}
        + \frac{\alpha}{\delta+2}
        \left(
            \begin{tikzpicture}[centerzero]
                \draw (-0.3,0.3) -- (-0.3,0.23) arc(180:360:0.15) -- (0,0.3);
                \draw (0.3,0.3) -- (0.15,0) -- (-0.2,-0.3);
                \draw (0.2,-0.3) -- (0.15,0);
            \end{tikzpicture}
            \! +\!
            \begin{tikzpicture}[centerzero]
                \draw (0.3,0.3) -- (0.3,0.23) arc(360:180:0.15) -- (0,0.3);
                \draw (-0.3,0.3) -- (-0.15,0) -- (0.2,-0.3);
                \draw (-0.2,-0.3) -- (-0.15,0);
            \end{tikzpicture}
            \! -
            \begin{tikzpicture}[centerzero]
                \draw (-0.2,-0.3) -- (-0.2,0.3);
                \draw (0,0.3) -- (0.15,0) -- (0.3,0.3);
                \draw (0.15,0) -- (0.15,-0.3);
            \end{tikzpicture}
            \! -\!
            \begin{tikzpicture}[centerzero]
                \draw (0.2,-0.3) -- (0.2,0.3);
                \draw (0,0.3) -- (-0.15,0) -- (-0.3,0.3);
                \draw (-0.15,0) -- (-0.15,-0.3);
            \end{tikzpicture}
            -\!
            \begin{tikzpicture}[centerzero]
                \draw (-0.15,-0.3) -- (-0.15,-0.23) arc(180:0:0.15) -- (0.15,-0.3);
                \draw (-0.3,0.3) -- (0,0.08) -- (0.3,0.3);
                \draw (0,0.3) -- (0,0.08);
            \end{tikzpicture}
            \! -\!
            \begin{tikzpicture}[centerzero]
                \draw (-0.3,0.3) -- (-0.15,0.15) -- (0,0.3);
                \draw (-0.15,0.15) -- (0.15,-0.3);
                \draw (0.3,0.3) -- (-0.15,-0.3);
            \end{tikzpicture}
            \! -\!
            \begin{tikzpicture}[centerzero]
                \draw (0.3,0.3) -- (0.15,0.15) -- (0,0.3);
                \draw (0.15,0.15) -- (-0.15,-0.3);
                \draw (-0.3,0.3) -- (0.15,-0.3);
            \end{tikzpicture}
            \! + 3
            \begin{tikzpicture}[centerzero]
                \draw (0,0.3) to[out=225,in=110] (-0.15,-0.3);
                \draw (0.3,0.3) -- (0,0) -- (-0.3,0.3);
                \draw (0,0) -- (0.15,-0.3);
            \end{tikzpicture}
            \! + 3
            \begin{tikzpicture}[centerzero]
                \draw (0,0.3) to[out=-45,in=70] (0.15,-0.3);
                \draw (-0.3,0.3) -- (0,0) -- (0.3,0.3);
                \draw (0,0) -- (-0.15,-0.3);
            \end{tikzpicture}
            \! - 3\,
            \begin{tikzpicture}[centerzero]
                \draw (0,0.3) -- (0,-0.15) -- (-0.15,-0.3);
                \draw (0,-0.15) -- (0.15,-0.3);
                \draw (-0.2,0.3) -- (-0.2,0.25) arc(180:360:0.2) -- (0.2,0.3);
            \end{tikzpicture}
        \right).
    \end{equation}
\end{lem}

\begin{proof}
    We have
    \begin{multline*}
        \begin{tikzpicture}[anchorbase]
            \draw (-0.2,0) -- (0,0.25) -- (0.2,0);
            \draw (0,0.25) -- (0,0.4);
            \draw (0.2,-0.25) -- (-0.2,0) -- (-0.3,0.4);
            \draw (-0.2,-0.25) -- (0.2,0) -- (0.3,0.4);
        \end{tikzpicture}
        \overset{\cref{venom}}{=}
        \begin{tikzpicture}[anchorbase]
            \draw (-0.2,0) -- (0,0.25) -- (0.2,0);
            \draw (0,0.25) -- (0,0.4);
            \draw (-0.2,0) -- (-0.3,0.4);
            \draw (-0.2,-0.25) -- (0.2,0) -- (0.3,0.4);
            \draw (-0.2,0) to[out=-45,in=180] (0.3,0.25) to[out=0,in=45] (0.2,-0.25);
        \end{tikzpicture}
        =
        \begin{tikzpicture}[anchorbase]
            \draw (-0.3,0.4) -- (-0.15,0.25) -- (0.3,-0.25);
            \draw (-0.15,0.25) -- (0,0.4);
            \draw (-0.15,0.25) -- (-0.2,0) -- (-0.3,-0.25);
            \draw (-0.2,0) -- (0.3,0.4);
            \opendot{-0.15,0.25};
        \end{tikzpicture}
        \overset{\cref{magic}}{=}
        \frac{2\alpha}{\delta+2}
        \left(
            \begin{tikzpicture}[centerzero]
                \draw (0,0.3) to[out=-45,in=70] (0.15,-0.3);
                \draw (-0.3,0.3) -- (0,0) -- (0.3,0.3);
                \draw (0,0) -- (-0.15,-0.3);
            \end{tikzpicture}
            +
            \begin{tikzpicture}[centerzero]
                \draw (-0.3,0.3) -- (-0.3,0.23) arc(180:360:0.15) -- (0,0.3);
                \draw (0.3,0.3) -- (0.15,0) -- (-0.2,-0.3);
                \draw (0.2,-0.3) -- (0.15,0);
            \end{tikzpicture}
            +
            \begin{tikzpicture}[centerzero]
                \draw (0.3,0.3) -- (0.15,0.15) -- (0,0.3);
                \draw (0.15,0.15) -- (-0.15,-0.3);
                \draw (-0.3,0.3) -- (0.15,-0.3);
            \end{tikzpicture}
        \right)
        -
        \begin{tikzpicture}[anchorbase]
            \draw (-0.3,0.4) -- (-0.2,0.2) -- (-0.05,0.2) -- (0,0.4);
            \draw (-0.2,0.2) -- (-0.2,-0.1) -- (-0.3,-0.25);
            \draw (-0.2,-0.1) -- (0.3,0.4);
            \draw (-0.05,0.2) -- (0.3,-0.25);
        \end{tikzpicture}
        -
        \begin{tikzpicture}[anchorbase]
            \draw (-0.3,0.4) -- (-0.1,0.2) -- (0,0.4);
            \draw (-0.1,0.2) -- (0,0) -- (0.3,0.4);
            \draw (-0.3,-0.25) -- (0,0) -- (0.3,-0.25);
            \opendot{0,0};
        \end{tikzpicture}
        \\
        \overset{\cref{magic}}{=}
        \frac{2\alpha}{\delta+2}
        \left(
            \begin{tikzpicture}[centerzero]
                \draw (0,0.3) to[out=-45,in=70] (0.15,-0.3);
                \draw (-0.3,0.3) -- (0,0) -- (0.3,0.3);
                \draw (0,0) -- (-0.15,-0.3);
            \end{tikzpicture}
            +
            \begin{tikzpicture}[centerzero]
                \draw (-0.3,0.3) -- (-0.3,0.23) arc(180:360:0.15) -- (0,0.3);
                \draw (0.3,0.3) -- (0.15,0) -- (-0.2,-0.3);
                \draw (0.2,-0.3) -- (0.15,0);
            \end{tikzpicture}
            +
            \begin{tikzpicture}[centerzero]
                \draw (0.3,0.3) -- (0.15,0.15) -- (0,0.3);
                \draw (0.15,0.15) -- (-0.15,-0.3);
                \draw (-0.3,0.3) -- (0.15,-0.3);
            \end{tikzpicture}
            -
            \begin{tikzpicture}[centerzero]
                \draw (0.2,-0.3) -- (0.2,0.3);
                \draw (0,0.3) -- (-0.15,0) -- (-0.3,0.3);
                \draw (-0.15,0) -- (-0.15,-0.3);
            \end{tikzpicture}
            -
            \begin{tikzpicture}[centerzero]
                \draw (-0.15,-0.3) -- (-0.15,-0.23) arc(180:0:0.15) -- (0.15,-0.3);
                \draw (-0.3,0.3) -- (0,0.08) -- (0.3,0.3);
                \draw (0,0.3) -- (0,0.08);
            \end{tikzpicture}
            -
            \begin{tikzpicture}[centerzero]
                \draw (-0.3,0.3) -- (-0.15,0.15) -- (0,0.3);
                \draw (-0.15,0.15) -- (0.15,-0.3);
                \draw (0.3,0.3) -- (-0.15,-0.3);
            \end{tikzpicture}
        \right)
        -
        \begin{tikzpicture}[anchorbase]
            \draw (-0.3,0.4) -- (-0.2,0.2) -- (-0.05,0.2) -- (0,0.4);
            \draw (-0.2,0.2) -- (-0.2,-0.1) -- (-0.3,-0.25);
            \draw (-0.2,-0.1) -- (0.3,0.4);
            \draw (-0.05,0.2) -- (0.3,-0.25);
        \end{tikzpicture}
        +
        \begin{tikzpicture}[anchorbase]
            \draw (0.3,-0.3) -- (-0.3,0.3);
            \draw (0.3,0.3) -- (0.15,-0.15);
            \draw (0,0.3) -- (-0.15,0.15);
            \draw (0,0) -- (-0.3,-0.3);
        \end{tikzpicture}
        +
        \begin{tikzpicture}[anchorbase]
            \draw (-0.3,0.3) -- (0,0) -- (0.3,0.3);
            \draw (0,0.3) -- (-0.15,0.15);
            \draw (0,0) -- (0,-0.15) -- (-0.15,-0.3);
            \draw (0,-0.15) -- (0.15,-0.3);
        \end{tikzpicture}
        .
    \end{multline*}
    Note that the third-to-last morphism appearing above is the rotation operator $\Rot$ applied to the first morphism.  Applying $1 - \Rot + \Rot^2 - \Rot^3 + \Rot^4$ to the above equation and dividing by $2$ then yields \cref{3spike}.
\end{proof}

\begin{lem} \label{pentexplode}
    If $\cI$ is a tensor ideal of $\Tcat$ satisfying \cref{cinco}, and $\delta \notin \{2,-6,10\}$, then the relation \cref{pentburst} holds in $\Tcat/\cI$.
\end{lem}

\begin{proof}
    We first show that we have a relation of the form
    \begin{equation} \label{pentbreak1}
        \begin{multlined}
            \pentmor =
            \gamma_1
            \left(
                \begin{tikzpicture}[anchorbase]
                    \draw (-0.2,0) -- (0,0.25) -- (0.2,0);
                    \draw (0,0.25) -- (0,0.4);
                    \draw (-0.2,-0.25) -- (-0.2,0) -- (-0.3,0.4);
                    \draw (0.2,-0.25) -- (0.2,0) -- (0.3,0.4);
                \end{tikzpicture}
                +
                \begin{tikzpicture}[anchorbase]
                    \draw (-0.3,0.3) -- (0,0) -- (0.3,0.3);
                    \draw (0,0.3) -- (-0.15,0.15);
                    \draw (0,0) -- (0,-0.15) -- (-0.15,-0.3);
                    \draw (0,-0.15) -- (0.15,-0.3);
                \end{tikzpicture}
                +
                \begin{tikzpicture}[anchorbase]
                    \draw (0.3,0.3) -- (0,0) -- (-0.3,0.3);
                    \draw (0,0.3) -- (0.15,0.15);
                    \draw (0,0) -- (0,-0.15) -- (0.15,-0.3);
                    \draw (0,-0.15) -- (-0.15,-0.3);
                \end{tikzpicture}
                +
                \begin{tikzpicture}[anchorbase]
                    \draw (-0.3,-0.3) -- (0.3,0.3);
                    \draw (-0.3,0.3) -- (-0.15,-0.15);
                    \draw (0,0.3) -- (0.15,0.15);
                    \draw (0,0) -- (0.3,-0.3);
                \end{tikzpicture}
                +
                \begin{tikzpicture}[anchorbase]
                    \draw (0.3,-0.3) -- (-0.3,0.3);
                    \draw (0.3,0.3) -- (0.15,-0.15);
                    \draw (0,0.3) -- (-0.15,0.15);
                    \draw (0,0) -- (-0.3,-0.3);
                \end{tikzpicture}
            \right)
            + \gamma_2
            \left(
                \begin{tikzpicture}[centerzero]
                    \draw (-0.15,-0.3) -- (-0.15,-0.23) arc(180:0:0.15) -- (0.15,-0.3);
                    \draw (-0.3,0.3) -- (0,0.08) -- (0.3,0.3);
                    \draw (0,0.3) -- (0,0.08);
                \end{tikzpicture}
                +
                \begin{tikzpicture}[centerzero]
                    \draw (-0.2,-0.3) -- (-0.2,0.3);
                    \draw (0,0.3) -- (0.15,0) -- (0.3,0.3);
                    \draw (0.15,0) -- (0.15,-0.3);
                \end{tikzpicture}
                +
                \begin{tikzpicture}[centerzero]
                    \draw (0.2,-0.3) -- (0.2,0.3);
                    \draw (0,0.3) -- (-0.15,0) -- (-0.3,0.3);
                    \draw (-0.15,0) -- (-0.15,-0.3);
                \end{tikzpicture}
                +
                \begin{tikzpicture}[centerzero]
                    \draw (-0.3,0.3) -- (-0.3,0.23) arc(180:360:0.15) -- (0,0.3);
                    \draw (0.3,0.3) -- (0.15,0) -- (-0.2,-0.3);
                    \draw (0.2,-0.3) -- (0.15,0);
                \end{tikzpicture}
                +
                \begin{tikzpicture}[centerzero]
                    \draw (0.3,0.3) -- (0.3,0.23) arc(360:180:0.15) -- (0,0.3);
                    \draw (-0.3,0.3) -- (-0.15,0) -- (0.2,-0.3);
                    \draw (-0.2,-0.3) -- (-0.15,0);
                \end{tikzpicture}
            \right)
            \\
            + \gamma_3
            \left(
                \begin{tikzpicture}[centerzero]
                    \draw (0,0.3) -- (0,-0.15) -- (-0.15,-0.3);
                    \draw (0,-0.15) -- (0.15,-0.3);
                    \draw (-0.2,0.3) -- (-0.2,0.25) arc(180:360:0.2) -- (0.2,0.3);
                \end{tikzpicture}
                +
                \begin{tikzpicture}[centerzero]
                    \draw (0,0.3) to[out=-45,in=70] (0.15,-0.3);
                    \draw (-0.3,0.3) -- (0,0) -- (0.3,0.3);
                    \draw (0,0) -- (-0.15,-0.3);
                \end{tikzpicture}
                +
                \begin{tikzpicture}[centerzero]
                    \draw (0,0.3) to[out=225,in=110] (-0.15,-0.3);
                    \draw (0.3,0.3) -- (0,0) -- (-0.3,0.3);
                    \draw (0,0) -- (0.15,-0.3);
                \end{tikzpicture}
                +
                \begin{tikzpicture}[centerzero]
                    \draw (-0.3,0.3) -- (-0.15,0.15) -- (0,0.3);
                    \draw (-0.15,0.15) -- (0.15,-0.3);
                    \draw (0.3,0.3) -- (-0.15,-0.3);
                \end{tikzpicture}
                +
                \begin{tikzpicture}[centerzero]
                    \draw (0.3,0.3) -- (0.15,0.15) -- (0,0.3);
                    \draw (0.15,0.15) -- (-0.15,-0.3);
                    \draw (-0.3,0.3) -- (0.15,-0.3);
                \end{tikzpicture}
            \right).
        \end{multlined}
    \end{equation}
    If the morphisms
    \begin{equation} \label{brutal}
        \begin{tikzpicture}[anchorbase]
            \draw (-0.2,0) -- (0,0.25) -- (0.2,0);
            \draw (0,0.25) -- (0,0.4);
            \draw (-0.2,-0.25) -- (-0.2,0) -- (-0.3,0.4);
            \draw (0.2,-0.25) -- (0.2,0) -- (0.3,0.4);
        \end{tikzpicture}
        \ ,\
        \begin{tikzpicture}[anchorbase]
            \draw (-0.3,0.3) -- (0,0) -- (0.3,0.3);
            \draw (0,0.3) -- (-0.15,0.15);
            \draw (0,0) -- (0,-0.15) -- (-0.15,-0.3);
            \draw (0,-0.15) -- (0.15,-0.3);
        \end{tikzpicture}
        \ ,\
        \begin{tikzpicture}[anchorbase]
            \draw (0.3,0.3) -- (0,0) -- (-0.3,0.3);
            \draw (0,0.3) -- (0.15,0.15);
            \draw (0,0) -- (0,-0.15) -- (0.15,-0.3);
            \draw (0,-0.15) -- (-0.15,-0.3);
        \end{tikzpicture}
        \ ,\
        \begin{tikzpicture}[anchorbase]
            \draw (-0.3,-0.3) -- (0.3,0.3);
            \draw (-0.3,0.3) -- (-0.15,-0.15);
            \draw (0,0.3) -- (0.15,0.15);
            \draw (0,0) -- (0.3,-0.3);
        \end{tikzpicture}
        \ ,\
        \begin{tikzpicture}[anchorbase]
            \draw (0.3,-0.3) -- (-0.3,0.3);
            \draw (0.3,0.3) -- (0.15,-0.15);
            \draw (0,0.3) -- (-0.15,0.15);
            \draw (0,0) -- (-0.3,-0.3);
        \end{tikzpicture}
        \ ,\
        \begin{tikzpicture}[centerzero]
            \draw (-0.15,-0.3) -- (-0.15,-0.23) arc(180:0:0.15) -- (0.15,-0.3);
            \draw (-0.3,0.3) -- (0,0.08) -- (0.3,0.3);
            \draw (0,0.3) -- (0,0.08);
        \end{tikzpicture}
        \ ,\
        \begin{tikzpicture}[centerzero]
            \draw (-0.2,-0.3) -- (-0.2,0.3);
            \draw (0,0.3) -- (0.15,0) -- (0.3,0.3);
            \draw (0.15,0) -- (0.15,-0.3);
        \end{tikzpicture}
        \ ,\
        \begin{tikzpicture}[centerzero]
            \draw (0.2,-0.3) -- (0.2,0.3);
            \draw (0,0.3) -- (-0.15,0) -- (-0.3,0.3);
            \draw (-0.15,0) -- (-0.15,-0.3);
        \end{tikzpicture}
        \ ,\
        \begin{tikzpicture}[centerzero]
            \draw (-0.3,0.3) -- (-0.3,0.23) arc(180:360:0.15) -- (0,0.3);
            \draw (0.3,0.3) -- (0.15,0) -- (-0.2,-0.3);
            \draw (0.2,-0.3) -- (0.15,0);
        \end{tikzpicture}
        \ ,\
        \begin{tikzpicture}[centerzero]
            \draw (0.3,0.3) -- (0.3,0.23) arc(360:180:0.15) -- (0,0.3);
            \draw (-0.3,0.3) -- (-0.15,0) -- (0.2,-0.3);
            \draw (-0.2,-0.3) -- (-0.15,0);
        \end{tikzpicture}
        \ ,\
        \begin{tikzpicture}[centerzero]
            \draw (0,0.3) -- (0,-0.15) -- (-0.15,-0.3);
            \draw (0,-0.15) -- (0.15,-0.3);
            \draw (-0.2,0.3) -- (-0.2,0.25) arc(180:360:0.2) -- (0.2,0.3);
        \end{tikzpicture}
        \ ,\
        \begin{tikzpicture}[centerzero]
            \draw (0,0.3) to[out=-45,in=70] (0.15,-0.3);
            \draw (-0.3,0.3) -- (0,0) -- (0.3,0.3);
            \draw (0,0) -- (-0.15,-0.3);
        \end{tikzpicture}
        \ ,\
        \begin{tikzpicture}[centerzero]
            \draw (0,0.3) to[out=225,in=110] (-0.15,-0.3);
            \draw (0.3,0.3) -- (0,0) -- (-0.3,0.3);
            \draw (0,0) -- (0.15,-0.3);
        \end{tikzpicture}
        \ ,\
        \begin{tikzpicture}[centerzero]
            \draw (-0.3,0.3) -- (-0.15,0.15) -- (0,0.3);
            \draw (-0.15,0.15) -- (0.15,-0.3);
            \draw (0.3,0.3) -- (-0.15,-0.3);
        \end{tikzpicture}
        \ ,\
        \begin{tikzpicture}[centerzero]
            \draw (0.3,0.3) -- (0.15,0.15) -- (0,0.3);
            \draw (0.15,0.15) -- (-0.15,-0.3);
            \draw (-0.3,0.3) -- (0.15,-0.3);
        \end{tikzpicture}
    \end{equation}
    are linearly independent, then this follows immediately from \cref{cinco} and the fact that the pentagon is invariant under rotation.  So we suppose there is a nontrivial linear dependence relation involving the morphisms \cref{brutal}.  If the coefficient of any of the first five morphisms is nonzero, we can rotate to assume that the coefficient of the first morphism is nonzero.  We then compose on the bottom with $\Hmor$ to obtain a pentagon from the first term.  Using \cref{chess,triangle,sqburst,3spike}, the other terms can be written as linear combinations of the morphisms in \cref{brutal}.  Summing over all rotations and solving for the pentagon then yields a relation of the form \cref{pentbreak1}.

    It remains to consider a nontrivial linear dependence relation involving the last 10 morphisms in \cref{brutal}.  Note that we can use rotation and crossings to turn any of these 10 morphisms into any other.  Thus, we may assume that the sixth morphism in \cref{brutal} occurs with coefficient $1$.  Then, applying symmetrizers to the bottom two strands and the top three strands, we obtain a relation of the form
    \begin{equation} \label{saidin}
        0 =
        \begin{tikzpicture}[centerzero]
            \draw (-0.15,-0.3) -- (-0.15,-0.23) arc(180:0:0.15) -- (0.15,-0.3);
            \draw (-0.3,0.3) -- (0,0.08) -- (0.3,0.3);
            \draw (0,0.3) -- (0,0.08);
        \end{tikzpicture}
        + \kappa_1
        \left(
            \begin{tikzpicture}[centerzero]
                \draw (-0.2,-0.3) -- (-0.2,0.3);
                \draw (0,0.3) -- (0.15,0) -- (0.3,0.3);
                \draw (0.15,0) -- (0.15,-0.3);
            \end{tikzpicture}
            +
            \begin{tikzpicture}[centerzero]
                \draw (0.2,-0.3) -- (0.2,0.3);
                \draw (0,0.3) -- (-0.15,0) -- (-0.3,0.3);
                \draw (-0.15,0) -- (-0.15,-0.3);
            \end{tikzpicture}
            +
            \begin{tikzpicture}[centerzero]
                \draw (0,0.3) to[out=-45,in=70] (0.15,-0.3);
                \draw (-0.3,0.3) -- (0,0) -- (0.3,0.3);
                \draw (0,0) -- (-0.15,-0.3);
            \end{tikzpicture}
            +
            \begin{tikzpicture}[centerzero]
                \draw (0,0.3) to[out=225,in=110] (-0.15,-0.3);
                \draw (0.3,0.3) -- (0,0) -- (-0.3,0.3);
                \draw (0,0) -- (0.15,-0.3);
            \end{tikzpicture}
            +
            \begin{tikzpicture}[centerzero]
                \draw (-0.3,0.3) -- (-0.15,0.15) -- (0,0.3);
                \draw (-0.15,0.15) -- (0.15,-0.3);
                \draw (0.3,0.3) -- (-0.15,-0.3);
            \end{tikzpicture}
            +
            \begin{tikzpicture}[centerzero]
                \draw (0.3,0.3) -- (0.15,0.15) -- (0,0.3);
                \draw (0.15,0.15) -- (-0.15,-0.3);
                \draw (-0.3,0.3) -- (0.15,-0.3);
            \end{tikzpicture}
        \right)
        +
        \kappa_2
        \left(
            \begin{tikzpicture}[centerzero]
                \draw (-0.3,0.3) -- (-0.3,0.23) arc(180:360:0.15) -- (0,0.3);
                \draw (0.3,0.3) -- (0.15,0) -- (-0.2,-0.3);
                \draw (0.2,-0.3) -- (0.15,0);
            \end{tikzpicture}
            +
            \begin{tikzpicture}[centerzero]
                \draw (0.3,0.3) -- (0.3,0.23) arc(360:180:0.15) -- (0,0.3);
                \draw (-0.3,0.3) -- (-0.15,0) -- (0.2,-0.3);
                \draw (-0.2,-0.3) -- (-0.15,0);
            \end{tikzpicture}
            +
            \begin{tikzpicture}[centerzero]
                \draw (0,0.3) -- (0,-0.15) -- (-0.15,-0.3);
                \draw (0,-0.15) -- (0.15,-0.3);
                \draw (-0.2,0.3) -- (-0.2,0.25) arc(180:360:0.2) -- (0.2,0.3);
            \end{tikzpicture}
        \right).
    \end{equation}
    Composing on the bottom with $\cupmor$, and using the fact that $\mergemor \ne 0$ (since this would imply the category is trivial, as noted at the end of \cref{sec:defin}), we get $\kappa_1=-\frac{\delta}{6}$.  Then adding a
    $
    \begin{tikzpicture}[anchorbase]
        \draw (-0.15,-0.15) -- (-0.15,0) arc(180:0:0.15) -- (0.15,-0.15);
        \draw (0.25,-0.15) -- (0.25,0.2);
    \end{tikzpicture}
    $
    to the top of \cref{saidin} implies that $\kappa_2 = -\frac{4 \kappa_1}{\delta+2} = \frac{2 \delta}{3(\delta+2)}$.  Finally, adding $\idstrand \, \mergemor$ to the top of \cref{saidin} and using \cref{magic} gives
    \[
        0 = - \left( \frac{\delta \alpha}{6} + \frac{2 \alpha \delta}{3(\delta+2)} \right)\, \jail
        + \left( \alpha - \frac{2 \alpha \delta}{3(\delta+2)} \right) \hourglass
        + \left( \frac{4\delta}{3(\delta+2)} + \frac{\delta}{3} \right) \Imor
        - \left( \frac{2 \alpha \delta}{3(\delta+2)} + \frac{\alpha \delta}{6} \right) \crossmor.
    \]
    By \cref{SUP}, all the above coefficients are zero.  However, this occurs if and only if $\delta=-6$.

    So we now have a relation of the form \cref{pentbreak1}.  Let $\theta = \frac{\alpha(2-\delta)}{2(\delta+2)}$, so that $\trimor = \theta \mergemor$.  Composing with $\mergemor$ on the rightmost two strings at the top of \cref{pentbreak1} gives
    \begin{multline*}
        \theta \sqmor
        =
        \gamma_1 \left( (\theta + \alpha) \left( \Hmor + \Imor \right) + \sqmor \right)
        + \gamma_2
        \left(
            \alpha \left(\, \jail + \hourglass\, \right)
            + \Hmor + \Imor
        \right)
        \\
        + \gamma_3 \left( \Hmor + \Imor + 2\, \dotcross + \alpha \crossmor \right).
    \end{multline*}
    Thus, using \cref{magic} to eliminate $\dotcross$, we have
    \[
        (\theta - \gamma_1) \sqmor
        =
        \left( \alpha \gamma_2 + \frac{4\alpha}{\delta+2} \gamma_3 \right)
        \left(\, \jail + \hourglass \, \right)
        + \left( (\theta+\alpha) \gamma_1 + \gamma_2 - \gamma_3 \right)
        \left(\, \Hmor + \Imor\, \right)
        + \frac{\alpha(\delta+6)}{\delta+2} \gamma_3
        \, \crossmor\ .
    \]
    Comparing to \cref{sqburst} and using the fact that the morphisms \cref{bigfive} are linearly independent (by \cref{SUP}), this gives
    \begin{align*}
        \frac{\alpha(\delta+14)}{2(\delta+2)^2} (\theta - \gamma_1)
        &= \gamma_2 + \frac{4}{\delta+2} \gamma_3,
        \\
        \frac{\alpha(\delta-6)}{2(\delta+2)} (\theta - \gamma_1)
        &= (\theta+\alpha)\gamma_1 + \gamma_2 - \gamma_3,
        \\
        \frac{3 \alpha (2-\delta)}{2(\delta+2)} (\theta - \gamma_1)
        &= (\delta+6) \gamma_3.
    \end{align*}
    Since $\delta \ne 2$, we can eliminate the left-hand sides of the first and third equation, eliminate the left-hand sides of the second and third equations, and substitute $\theta = \frac{\alpha(2-\delta)}{2(\delta+2)}$ to obtain the linear system
    \begin{gather} \nonumber
        \gamma_2 + \frac{\delta+30}{3(\delta-2)} \gamma_3
        = 0,
        \qquad
        \frac{\alpha(\delta+6)}{2(\delta+2)} \gamma_1 + \gamma_2 + \frac{\delta^2-3\delta-30}{3(\delta-2)} \gamma_3
        = 0,
        \\ \label{pent1}
        \frac{3 \alpha (2-\delta)}{2(\delta+2)} \gamma_1
        + (\delta+6) \gamma_3
        = \frac{3 \alpha^2 (2-\delta)^2}{4(\delta+2)^2}.
    \end{gather}
    Subtracting the first equation above from the second gives
    \[
        \frac{\alpha(\delta+6)}{2(\delta+2)} \gamma_1 + \frac{(\delta+6)(\delta-10)}{3(\delta-2)} \gamma_3
        = 0
        \iff
        \gamma_3 =
        \frac{3 \alpha (2-\delta)}{2(\delta+2)(\delta-10)} \gamma_1,
    \]
    since $\delta \ne -6,10$.  Combining with \cref{pent1} then gives
    \[
        \frac{3 \alpha (2-\delta)}{2(\delta+2)} \gamma_1
        + \frac{3 \alpha (\delta+6)(2-\delta)}{2(\delta+2)(\delta-10)} \gamma_1
        = \frac{3 \alpha^2 (2-\delta)^2}{4(\delta+2)^2}
        \implies
        \gamma_1
        = \frac{\alpha(10-\delta)}{4(\delta+2)}.
    \]
    Thus
    \[
        \gamma_3 =
        \frac{3 \alpha (2-\delta)}{2(\delta+2)(\delta-10)} \gamma_1
        = \frac{3 \alpha^2 (\delta-2)}{8(\delta+2)^2}
        \quad \text{and} \quad
        \gamma_2
        = \frac{\delta+30}{3(2-\delta)} \gamma_3
        = - \frac{\alpha^2 (\delta+30)}{8(\delta+2)^2},
    \]
    giving \cref{pentburst}.
    \details{
        As a check, adding a $\capmor$ to the rightmost top two strings in \cref{pentburst} gives
        \[
            \alpha \theta \mergemor
            = \gamma_1 (2\alpha + \theta) \mergemor
            + \gamma_2 (\delta + 2) \mergemor
            + 4 \gamma_3 \mergemor.
        \]
        So we should have
        \[
            \alpha \theta = (2 \alpha + \theta)\gamma_1 +(\delta+2)\gamma_2 + 4 \gamma_3,
        \]
        that is,
        \[
            \frac{\alpha^2(2-\delta)}{2(\delta+2)}
            = \frac{\alpha(3\delta+10)}{2(\delta+2)} \gamma_1 + (\delta+2) \gamma_2 + 4 \gamma_3.
        \]
        Indeed, we have
        \begin{align*}
            \frac{\alpha(3\delta+10)}{2(\delta+2)} \gamma_1 + (\delta+2) \gamma_2 + 4 \gamma_3
            &= \frac{\alpha(3\delta+10)}{2(\delta+2)} \frac{\alpha(10-\delta)}{4(\delta+2)}
            - (\delta+2) \frac{\alpha^2 (\delta+30)}{8(\delta+2)^2}
            + 4 \frac{3 \alpha^2 (\delta-2)}{8(\delta+2)^2}
            \\
            &= \frac{\alpha^2 (16-4\delta^2)}{8(\delta+2)^2}
            = \frac{\alpha^2 (2-\delta)}{2(\delta+2)}
        \end{align*}
        as desired.
    }
\end{proof}

\begin{rem}
    When $\delta \in \{2,-6,10\}$, there exist other solutions to the linear system appearing in the proof of \cref{pentexplode}.  We are not sure of the role of these other categories in representation theory.  Compare to \cref{zagreb}.
\end{rem}

\begin{proof}[Proof of \cref{demayo}]
    The theorem now follows immediately from \cref{SUP,sqexplode,pentexplode}.
\end{proof}

\begin{rem}
    When combined with \cref{meow}, \cref{demayo} implies that every pivotal symmetric monoidal category $\cC$ generated by a symmetric self-dual object $\go$ and a rotationally invariant symmetric morphism $\one \to \go^{\otimes 3}$, and with $\dim \cC(\one, \go^{\otimes n})$ equal to $1,0,1,1,5,15$ for $n=0,1,2,3,4,5$, respectively, is a quotient of $\Fcat$ for some value of $\delta$.  Similar categories, with different conditions on the dimensions, were classified in \cite{MPS17}.  The corresponding statement for the quantum $G_2$ link invariant is given in \cite[Th.~2.1]{Kup94}.
\end{rem}

\section{The Albert algebra and the Lie group of type $F_4$}

In this section, we will develop some properties of the Albert algebra and the Lie group and Lie algebra of type $F_4$ that will be used in the sequel. For further details, we refer the reader to \cite[Ch.~16]{Ada96}.  Let
\[
    A
    =
    \left\{
        \begin{pmatrix}
            \lambda_1 & x_3 & \bar{x}_2 \\
            \bar{x}_3 & \lambda_2 & x_1 \\
            x_2 & \bar{x}_1 & \lambda_3
        \end{pmatrix}
        : \lambda_i \in \R,\ x_i \in \OO
    \right\}
\]
denote the set of $3 \times 3$ self-adjoint matrices over the octonions $\OO$, equipped with the bilinear operation
\[
    a \circ b := \frac{1}{2}(ab+ba),\quad a,b \in A,
\]
where the juxtaposition $ab$ denotes usual matrix multiplication.  Thus $A$ is one of the three real Albert algebras.  Note that this algebra is commutative and unital, but \emph{not} associative.  We have $\dim_\R(A) = 27$.  Eventually, we will be interested in the complexification $\C \otimes_\R A$, which is the unique simple exceptional complex Jordan algebra, up to isomorphism.  However, since many of our preliminary arguments are valid over $\R$, we state them in that setting.

Let $\tr \colon A \to \R$ denote the trace map, so that
\begin{equation} \label{Atrace}
    \tr
    \begin{pmatrix}
        \lambda_1 & x_3 & \bar{x}_2 \\
        \bar{x}_3 & \lambda_2 & x_1 \\
        x_2 & \bar{x}_1 & \lambda_3
    \end{pmatrix}
    = \lambda_1 + \lambda_2 + \lambda_3.
\end{equation}
For $a \in A$, let
\[
    L_a \colon A \to A,\quad b \mapsto a \circ b,
\]
denote the $\R$-linear map given by left multiplication by $a$.

\begin{lem}
    For $a \in A$, we have
    \begin{equation} \label{slime}
        \tr(a) = \tfrac{1}{9} \Tr(L_a),
    \end{equation}
    where $\Tr(L_a)$ denotes the usual trace of the linear operator $L_a$ on the $27$-dimensional real vector space $A$.
\end{lem}

\begin{proof}
    Let $E_{ij}$ denote the matrix with a $1$ in the $(i,j)$-entry and all other entries equal to zero.  Since both sides of \cref{slime} are $\R$-linear in $a$, it suffices to consider the cases where $a = x E_{ij} + \bar{x} E_{ji}$ for $x \in \OO$, $1 \le i,j \le 3$.  Consider the basis of $A$ given by the elements
    \[
        y E_{kl} + \bar{y} E_{lk},\quad k \le l,
    \]
    where $y$ runs over a basis of $\OO$ if $k \ne l$, and $y=\frac{1}{2}$ if $k=l$.  Now
    \[
        L_a(y E_{kl} + \bar{y} E_{lk})
        = \frac{1}{2}
        \left(
            \delta_{jk} xy E_{il}
            + \delta_{jl} x\bar{y} E_{ik}
            + \delta_{ik} \bar{x}y E_{jl}
            + \delta_{il} \bar{x}\bar{y} E_{jk}
        \right).
    \]
    We see that the $yE_{kl} + \bar{y}E_{lk}$ component of this is zero unless $i=j=k$ or $i=j=l$.  Thus $\tr(L_a)$ is zero unless $a$ is a diagonal matrix.  On the other hand, if $a = \lambda E_{ii}$, $\lambda \in \R$, then $L_a$ acts as multiplication by $\frac{\lambda}{2}(\delta_{ik} + \delta_{il})$ on the subspace $x E_{kl} + \bar{x} E_{lk}$, $x \in \OO$, $k \ne l$, and multiplication by $\lambda$ on the subspace $\R E_{ii}$.  Thus $\Tr(L_a) = (1 + \frac{1}{2}(16)) \lambda = 9 \lambda$.
\end{proof}

Let $G$ denote the group of algebra automorphisms of $A$.  Thus $G$ is the compact connected real Lie group of type $F_4$.

\begin{lem} \label{versa}
    We have $\tr(ga) = \tr(a)$ for all $a \in A$ and $g \in G$.
\end{lem}

\begin{proof}
    For $a \in A$ and $g \in G$, we have
    \[
        (g L_a g^{-1})(b)
        = g(a \circ (g^{-1}b))
        = (ga) \circ b
        = L_{ga}(b).
    \]
    Thus $\tr(ga) = \frac{1}{9} \Tr(L_{ga}) = \frac{1}{9} \Tr (g L_a g^{-1}) = \frac{1}{9} \Tr(L_a) = \tr(a)$.
\end{proof}

\begin{cor} \label{squirrel}
    The symmetric bilinear form
    \begin{equation} \label{Bdef}
        B \colon A \otimes A \to \R,\quad
        B(a \otimes b) := \tr(a \circ b),
    \end{equation}
    is nondegenerate and $G$-invariant.
\end{cor}

We will sometimes write $B(a,b)$ for $B(a \otimes b)$.

\begin{proof}
    Direct computation shows that, if $a$ is the matrix appearing in \cref{Atrace}, then
    \[
        B(a, a)
        = \sum_{i=1}^3 \left( \lambda_i^2 + 2 \|x_i\|^2 \right),
    \]
    which is nonzero when $a \ne 0$.  Hence $B$ is nondegenerate.  Since
    \[
        B(ga, gb)
        = \tr((ga) \circ (gb))
        = \tr(g(a \circ b))
        = \tr(a \circ b)
        = B(a \otimes b),
    \]
    we also see that $B$ is $G$-invariant.
\end{proof}

It follows from \cref{squirrel} that we have a decomposition of $G$-modules
\begin{equation} \label{rabbit}
    A = \R 1_A \oplus V_\R,\quad
    V_\R := \ker(\tr).
\end{equation}
Then $V_\R$ is the $26$-dimensional irreducible $G$-module (\cite[Cor.~16.2]{Ada96}).  Let
\begin{equation} \label{blueberry}
    \pi \colon A \to V_\R,\quad \pi(a) = a - \frac{1}{3} \tr(a) 1_A
\end{equation}
be the projection along the decomposition \cref{rabbit}.

For $x \in \OO$, let $\RP(x)$ denote its real part.  It is straightforward to verify that (or see \cite[Lem.~15.9, Cor.~15.12]{Ada96})
\begin{equation} \label{crayon}
    \RP(xy) = \RP(yx),\quad
    \RP((xy)z) = \RP(x(yz)),\qquad
    x,y,z \in \OO.
\end{equation}
For $X \in \Mat_{3 \times 3}(\OO)$, let $\tr_\R = \RP(\tr(X))$.

\begin{lem}
    For $X,Y,Z \in \Mat_{3 \times 3}(\OO)$, we have
    \begin{equation} \label{trip}
        \tr_\R(XY) = \tr_\R(YX)
        \quad \text{and} \quad
        \tr_\R((XY)Z) = \tr_\R(X(YZ)).
    \end{equation}
\end{lem}

\begin{proof}
    Since both sides of both equalities to be proved are $\R$-linear in $X$, $Y$, and $Z$, it suffices to consider the case where $X = x E_{ij}$, $Y = y E_{kl}$, and $Z = z E_{mn}$, for $x,y,z \in \OO$.  For the first equality, we have
    \[
        \tr_\R(XY)
        = \delta_{jk} \delta_{il} \RP(xy)
        \overset{\cref{crayon}}{=} \delta_{jk} \delta_{il} \RP(yx)
        = \tr_\R(YX).
    \]
    For the second equality, we have
    \[
        \tr_\R((XY)Z)
        = \delta_{jk} \delta_{lm} \delta_{in} \RP((xy)z)
        \overset{\cref{crayon}}{=} \delta_{jk} \delta_{lm} \delta_{in} \RP(x(yz))
        = \tr_\R(X(YZ)). \qedhere
    \]
\end{proof}

\begin{lem}
    We have
    \begin{equation} \label{cedar}
        \tr((a \circ b) \circ c) = \tr(a \circ (b \circ c)),\quad a,b,c \in A.
    \end{equation}
\end{lem}

\begin{proof}
    For $a,b,c \in A$, we have
    \begin{multline*}
        4 \tr((a \circ b) \circ c)
        = 4 \tr_\R((a \circ b) \circ c)
        = \tr_\R((ab)c + (ba)c + c(ab) + c(ba)) \\
        \overset{\cref{trip}}{=} \tr_\R(a(bc) + a(cb) + (bc)a + (cb)a)
        = 4 \tr_\R(a \circ (b \circ c))
        = 4 \tr(a \circ (b \circ c)). \qedhere
    \end{multline*}
\end{proof}

\begin{lem} \label{raspberry}
    For $a \in V_\R$, we have
    \begin{equation} \label{boysenberry}
        \pi(\pi(a \circ a) \circ a) = \tfrac{1}{6} \tr(a \circ a) a.
    \end{equation}
\end{lem}

\begin{proof}
    Since $\tr(a)=0$, \cite[Th.~16.6(iii)]{Ada96} (which is essentially the Cayley--Hamilton theorem for $A$) implies that
    \[
        0 = (a \circ a) \circ a - \tfrac{1}{2} \tr(a \circ a) a - \tfrac{1}{3} \tr((a \circ a) \circ a)
        \overset{\cref{blueberry}}{=} \pi((a \circ a) \circ a) - \tfrac{1}{2} \tr(a \circ a) a.
    \]
    Note that the $\ell$, $b$, and $t$ of \cite{Ada96} are our $\tr$, $B$, and $\tr((- \circ -) \circ -)$, respectively.  Since $\pi(a)=a$, we also have
    \[
        \pi((a \circ a) \circ a)
        \overset{\cref{blueberry}}{=} \pi(\pi(a \circ a) \circ a) + \tfrac{1}{3} \tr(a \circ a) a.
    \]
    The identity \cref{boysenberry} now follows.
\end{proof}

\begin{cor}
    For $a,b,c \in V_\R$, we have
    \begin{equation} \label{mango}
        \pi(\pi(a \circ b) \circ c) + \pi(\pi(b \circ c) \circ a) + \pi(\pi(a \circ c) \circ b)
        = \tfrac{1}{6} \left( \tr(b \circ c) a + \tr(a \circ b) c + \tr(a \circ c) b \right).
    \end{equation}
\end{cor}

\begin{proof}
    This is the polarization of the identity \cref{boysenberry}.  That is, we replace $a$ in \cref{boysenberry} by $\lambda_a a + \lambda_b b + \lambda_c c$, expand, take the $\lambda_a \lambda_b \lambda_c$ terms, and then divide by $2$.
\end{proof}

Fix a basis $\B_V$ of $V_\R$.  By \cref{squirrel}, together with the fact that $\tr(V_\R \circ 1_A)=0$, there exists a dual basis $\B_V^\vee = \{b^\vee : b \in \B_V\}$ defined by
\[
    \tr(a^\vee \circ b) = \delta_{a,b},\quad a,b \in \B_V.
\]
We can extend $\B_V$ to a basis $\B_A := \B_V \sqcup \{1_A\}$ of $A$, with dual basis $\B_V^\vee \sqcup \{\frac{1}{3} 1_A\}$, i.e.\ $1^\vee = \frac{1}{3} 1_A$.  Note that the elements $\sum_{b \in \B_V} b \otimes b^\vee \in V_\R \otimes V_\R$ and $\sum_{b \in \B_A} b \otimes b^\vee \in A \otimes A$ are both independent of the choice of bases.

\begin{lem}
    For $a \in A$, we have
    \begin{equation} \label{teleport}
        \sum_{b \in \B_A} a \circ b \otimes b^\vee
        = \sum_{b \in \B_A} b \otimes b^\vee \circ a.
    \end{equation}
\end{lem}

\begin{proof}
    We have
    \[
        \sum_{b \in \B_A} a \circ b \otimes b^\vee
        = \sum_{b,c \in \B_A} \tr(c^\vee \circ (a \circ b)) c \otimes b^\vee
        \overset{\cref{cedar}}{=} \sum_{b,c \in \B_A} c \otimes \tr((c^\vee \circ a) \circ b) b^\vee
        = \sum_{c \in \B_A} c \otimes c^\vee \circ a. \qedhere
    \]
\end{proof}

Let $\fg = \C \otimes_\R \fg_\R$ be the complexification of the Lie algebra $\fg_\R$ of $G$, and let $V = \C \otimes_\R V_\R$ be the corresponding natural $\fg$-module.  Let $\fg$-mod denote the category of finite-dimensional $\fg$-modules.  We continue to denote by $\tr$ and $B$ the complexification of the maps \cref{Atrace,Bdef}.  Then $\B_V$ is also a $\C$-basis of $V$ with dual basis $\B_V^\vee$.  We will continue to use the bar notation $\bar{\ }$ to denote the conjugation of the octonions, extended to their complexification by $\C$-linearity.

We conclude this section by recalling some basic facts about the representation theory of $\fg$. Consider the following labeling of the nodes of the Dynkin diagram of type $F_4$:
\[
    \begin{tikzpicture}[centerzero]
        \draw (0,0) -- (1,0);
        \draw (2,0) -- (3,0);
        \draw[style=double,double distance=2pt] (1,0) -- (2,0);
        \draw[style=double,double distance=2pt,-{Classical TikZ Rightarrow[length=3mm,width=4mm]}] (1,0) -- (1.65,0);
        \filldraw (0,0) circle (2pt) node[anchor=south] {$1$};
        \filldraw (1,0) circle (2pt) node[anchor=south] {$2$};
        \filldraw (2,0) circle (2pt) node[anchor=south] {$3$};
        \filldraw (3,0) circle (2pt) node[anchor=south] {$4$};
    \end{tikzpicture}
\]
Let $\omega_1,\omega_2,\omega_3,\omega_4$ denote the corresponding fundamental weights.  For a dominant integral weight $\lambda$, let $V_\lambda$ denote the simple $\fg$-module of highest weight $\lambda$.  In particular $V = V_{\omega_4}$, while $V_{\omega_1}$ is the adjoint representation.  We have tensor product decompositions
\begin{align} \label{2decomp}
    V^{\otimes 2} &= V_0 \oplus V_{\omega_1} \oplus V_{\omega_3} \oplus V_{\omega_4} \oplus V_{2 \omega_4},
    \\ \label{3decomp}
    V^{\otimes 3} &= V_0 \oplus V_{\omega_1}^{\oplus 2} \oplus V_{\omega_2} \oplus V_{\omega_3}^{\oplus 4} \oplus V_{\omega_4}^{\oplus 5} \oplus V_{\omega_1+\omega_4}^{\oplus 3} \oplus V_{\omega_3+\omega_4}^{\oplus 2} \oplus V_{2\omega_4}^{\oplus 3} \oplus V_{3 \omega_4}.
\end{align}
This follows, for example, from the table given in \cite[Ch.~11,~Table~7]{MPR90}.  By Schur's lemma, we thus have
\begin{equation} \label{measure}
    \dim \Hom_\fg(V^{\otimes 2}, V^{\otimes 2}) = 5
    \quad \text{and} \quad
    \dim \Hom_\fg(V^{\otimes 3}, V^{\otimes 2}) = 15
\end{equation}
The importance of these dimensions is the assumption \cref{cinco} in \cref{demayo}.

\section{The functor\label{sec:functor}}

In this section we describe a natural functor from the category $\Fcat$ to the category $\fg$-mod of finite-dimensional modules over the complex Lie algebra $\fg$ of type $F_4$.  We do this by first defining a functor from $\Tcat$, and then showing that it factors through $\Fcat$.  Throughout this section we work over the field $\kk = \C$.

\begin{theo} \label{magneto}
    There is a unique monoidal functor
    \[
        \Phi \colon \Tcat_{7/3,26} \to \fg\md
    \]
    given on objects by $\go \mapsto V$ and on morphisms by
    \begin{align}
        \Phi(\mergemor) &\colon V \otimes V \to V,&
        a \otimes b &\mapsto \pi(a \circ b),
        \\
        \Phi(\crossmor) &\colon V \otimes V \to V \otimes V,&
        a \otimes b &\mapsto b \otimes a,
        \\
        \Phi(\cupmor) &\colon \C \to V \otimes V,&
        1 &\mapsto \sum_{b \in \B_V} b \otimes b^\vee,
        \\
        \Phi(\capmor) &\colon V \otimes V \to \C,&
        a \otimes b &\mapsto B(a \otimes b) = \tr(a \circ b).
    \end{align}
    Furthermore
    \begin{equation} \label{Gsplit}
        \Phi \left( \splitmor \right) \colon V \to V \otimes V,\qquad
        a \mapsto
        \sum_{b \in \B_V} b \otimes \pi (b^\vee \circ a)
        =
        \sum_{b \in \B_V} \pi(a \circ b) \otimes b^\vee.
    \end{equation}
\end{theo}

\begin{proof}
    We must verify that $\Phi$ respects the defining relations in \cref{Tdef}.  For the third equality in \cref{vortex}, we first use \cref{teleport} to see that, for $a \in V$,
    \[
        \sum_{b \in \B_V} a \circ b \otimes b^\vee + a \otimes \tfrac{1}{3}
        = \sum_{b \in \B_V} b \otimes b^\vee \circ a + \tfrac{1}{3} \otimes a.
    \]
    Applying $\pi \otimes \pi$ yields
    \[
        \sum_{b \in \B_V} \pi(a \circ b) \otimes b^\vee
        = \sum_{b \in \B_V} b \otimes \pi(b^\vee \circ a).
    \]
    Thus, for $a \in V$, we have
    \[
        \Phi
        \left(
            \begin{tikzpicture}[anchorbase]
                \draw (-0.4,0.2) to[out=down,in=180] (-0.2,-0.2) to[out=0,in=225] (0,0);
                \draw (0,0) -- (0,0.2);
                \draw (0.3,-0.3) -- (0,0);
            \end{tikzpicture}
        \right)
        (a)
        =
        \sum_{b \in \B_V} b \otimes \pi (b^\vee \circ a)
        =
        \sum_{b \in \B_V} \pi(a \circ b) \otimes b^\vee
        =
        \Phi
        \left(
            \begin{tikzpicture}[anchorbase]
                \draw (0.4,0.2) to[out=down,in=0] (0.2,-0.2) to[out=180,in=-45] (0,0);
                \draw (0,0) -- (0,0.2);
                \draw (-0.3,-0.3) -- (0,0);
            \end{tikzpicture}
        \right)
        (a).
    \]
    This shows that $\Phi$ preserves the third equality in \cref{vortex} and that it satisfies \cref{Gsplit}.  The verification of the relations \cref{venom} and the first two equalities in \cref{vortex} are now straightforward.

    For the fourth equality in \cref{vortex}, we compute
    \[
        \Phi
        \left(
            \begin{tikzpicture}[centerzero]
                \draw (-0.2,-0.3) -- (-0.2,-0.1) arc(180:0:0.2) -- (0.2,-0.3);
                \draw (-0.3,0.3) \braiddown (0,-0.3);
            \end{tikzpicture}
        \right)
        (a \otimes b \otimes c)
        = B(a \otimes c) b
        = \Phi
        \left(
            \begin{tikzpicture}[centerzero]
                \draw (-0.2,-0.3) -- (-0.2,-0.1) arc(180:0:0.2) -- (0.2,-0.3);
                \draw (0.3,0.3) \braiddown (0,-0.3);
            \end{tikzpicture}
        \right)
        (a \otimes b \otimes c).
    \]

    The fact that $\Phi$ respects the first two relations in \cref{chess} follows immediately from the fact that $a \circ b = b \circ a$ for $a,b \in V$.  Now consider the third relation in \cref{chess}.      Since $V$ is an irreducible $\fg$-module, there exists $\alpha \in \C$ such that
    \[
        \Phi
        \left(
            \begin{tikzpicture}[centerzero]
                \draw  (0,-0.4) -- (0,-0.2) to[out=45,in=down] (0.15,0) to[out=up,in=-45] (0,0.2) -- (0,0.4);
                \draw (0,-0.2) to[out=135,in=down] (-0.15,0) to[out=up,in=-135] (0,0.2);
            \end{tikzpicture}
        \right)
        (a)
        = \alpha a
        \quad \text{for all } a \in V,
    \]
    and we must show that $\alpha = \tfrac{7}{3}$.  It suffices to show this for some specific choice of $a$, so we choose $a = E_{11} - E_{22}$.  Now, choose the basis
    \begin{align*}
        \B_V &= \{b_1 = \tfrac{1}{\sqrt{2}} E_{11} - \tfrac{1}{\sqrt{2}} E_{22},\, b_2 = \tfrac{1}{\sqrt{6}} E_{11} + \tfrac{1}{\sqrt{6}} E_{22} - \tfrac{2}{\sqrt{6}} E_{33}\} \sqcup \B_V',\quad \text{where}
        \\
        \B_V' &= \{ \tfrac{1}{\sqrt{2}}(x E_{ij} + \bar{x} E_{ji}) : 1 \le i < j \le 3,\ x \in \B_\OO\},
    \end{align*}
    and where $\B_\OO$ is the usual basis of unit octonions (in particular, $\bar{x} x = x \bar{x} = 1$ for $x \in \B_\OO$).  Then $\B_V$ is an orthonormal basis for $V$, that is, $b^\vee = b$ for all $b \in \B_V$.  Furthermore, for $b = \frac{1}{\sqrt{2}}(xE_{ij} + \bar{x}E_{ji})$, with $x \in \B_\OO$ and $1 \le i < j \le 3$, we have
    \[
        (a \circ b) \circ b^\vee
        = \tfrac{1}{2}(\delta_{i1} - \delta_{i2} - \delta_{j2}) b \circ b
        = \tfrac{1}{4}(\delta_{i1} - \delta_{i2} - \delta_{j2})(E_{ii} + E_{jj}).
    \]
    Therefore,
    \begin{multline*}
        \sum_{b \in \B_V} (a \circ b) \circ b^\vee
        = (a \circ b_1) \circ b_1 + (a \circ b_2) \circ b_2 + \sum_{b \in \B_V'} (a \circ b) \circ b
        \\
        = \tfrac{1}{2}(E_{11}-E_{22}) + \tfrac{1}{6}(E_{11}-E_{22})
        + 2 \sum_{i<j} (\delta_{i1}-\delta_{i2}-\delta_{j2}) (E_{ii}+E_{jj})
        = \tfrac{8}{3}a.
    \end{multline*}
    Thus we have
    \begin{multline*}
        \Phi
        \left(
            \begin{tikzpicture}[centerzero]
                \draw  (0,-0.4) -- (0,-0.2) to[out=45,in=down] (0.15,0) to[out=up,in=-45] (0,0.2) -- (0,0.4);
                \draw (0,-0.2) to[out=135,in=down] (-0.15,0) to[out=up,in=-135] (0,0.2);
            \end{tikzpicture}
        \right)
        (a)
        = \sum_{b \in \B_V} \pi \left( \pi(a \circ b) \circ b^\vee \right)
        \overset{\cref{blueberry}}{=} \sum_{b \in \B_V} \pi((a \circ b) \circ b^\vee) - \tfrac{1}{3} \sum_{b \in \B_V} B(a \otimes b) \pi(b^\vee)
        \\
        = \sum_{b \in \B_V} \pi((a \circ b) \circ b^\vee) - \tfrac{1}{3} \pi(a)
        = \tfrac{8}{3} a - \tfrac{1}{3} a
        = \tfrac{7}{3} a,
    \end{multline*}
    as desired.

    That $\Phi$ respects the fourth equality in \cref{chess} follows immediately from the fact that $\dim_\C V = 26$.  Since $\Phi \left( \lollydrop \right)$ is a homomorphism of $\fg$-modules from the trivial module to $V$, it must be zero by Schur's lemma, proving that $\Phi$ respects the fifth equality in \cref{chess}.
\end{proof}

For a linear category $\cC$, let $\Kar(\cC)$ denote its additive Karoubi envelope.  Thus, objects of $\Kar(\cC)$ are pairs $(X,e)$, where $X$ is an object in the additive envelope $\Add(\cC)$ of $\cC$, and $e \in \cC(X,X)$ is an idempotent endomorphism.  Morphisms in $\Kar(\cC)$ are given by
\[
    \Kar(\cC) \big( (X,e),(X',e') \big) = e' \Add(\cC)(X,X') e.
\]
Composition is as in $\cC$.

\begin{prop} \label{FunctorFull}
    The functor $\Phi$ is full.
\end{prop}

\begin{proof}
    We follow the method of the proof of \cite[Th.~5.1]{Kup96}.  Since the category $\fg$-mod is idempotent complete, we have an induced functor
    \begin{equation} \label{river}
        \Kar(\Tcat_{7/3,26}) \to \fg\md.
    \end{equation}
    Let $\cC$ be the image of this functor.  Then $\cC$ is a rigid symmetric monoidal category.  Now consider the conjugate-linear contravariant monoidal endofunctor $\Xi$ of $\Fcat$ determined on objects by $\go \mapsto \go$ and on morphisms by
    \[
        \mergemor \mapsto \splitmor,\quad
        \crossmor \mapsto \crossmor,\quad
        \cupmor \mapsto \capmor,\quad
        \capmor \mapsto \cupmor.
    \]
    Intuitively, $\Xi$ is given by reflecting diagrams in the vertical axis and taking the complex conjugate of all coefficients.  Then $\Phi$ intertwines $\Xi$ with the hermitian adjoint, and so it follows that $\End_\cC(V^{\otimes n})$ is closed under hermitian adjoint.  Hence $\cC$ satisfies the hypotheses of \cite[Th.~6.1]{DR89}, which is a category-theoretic analogue of Krein's theorem \cite[p.~177]{Kir76}.  Therefore, $\cC$ is the category of finite-dimensional representations of some compact group $H$.  Since all morphisms of $\cC$ are homomorphisms of $G$-modules, we have $G \subseteq H$.  On the other hand, $G$ is precisely the group of automorphisms of $V$ preserving $\Phi(\mergemor)$ and $\Phi(\capmor)$.  Thus $G=H$ and so \cref{river} is full.  Viewing $\Tcat$ as a full subcategory of $\Kar(\Tcat)$ in the usual way, we conclude that $\Phi$ is full.
\end{proof}

\begin{theo} \label{baja}
    The functor $\Phi$ from \cref{magneto} factors through $\Fcat_{7/3, 26}$.
\end{theo}

\begin{proof}
    Let $\cI$ be the kernel of the functor $\Phi$.  Then $\cI$ is a tensor ideal of $\Tcat$, and we must show that the relations \cref{magic,sqburst,pentburst} are satisfied in $\Tcat/\cI$.  By \cref{FunctorFull,measure}, it follows that condition \cref{cinco} is satisfied.  Then the theorem follows from \cref{demayo}.
\end{proof}

\begin{rem} \label{prestige}
    It is also possible to give a more direct proof that the image of \cref{magic} under $\Phi$ holds in $\fg$-mod.  Bending the top right endpoint to the bottom of the diagrams by tensoring on the right with $\go$ and attaching a cup to the two rightmost endpoints at the top of the diagram (an operation which is invertible by the first relation in \cref{vortex}), we see that \cref{magic} is equivalent to
    \[
        \begin{tikzpicture}[anchorbase]
            \draw (-0.4,-0.4) -- (0,0) -- (0,0.3);
            \draw (0,-0.4) -- (-0.2,-0.2);
            \draw (0.4,-0.4) -- (0,0);
        \end{tikzpicture}
        +
        \begin{tikzpicture}[anchorbase]
            \draw (-0.4,-0.4) -- (0,0) -- (0,0.3);
            \draw (0,-0.4) -- (0.2,-0.2);
            \draw (0.4,-0.4) -- (0,0);
        \end{tikzpicture}
        +
        \begin{tikzpicture}[anchorbase]
            \draw (-0.4,-0.4) -- (0,0) -- (0,0.3);
            \draw (0.4,-0.4) -- (-0.2,-0.2);
            \draw (0,-0.4) -- (0.2,-0.2) -- (0,0);
        \end{tikzpicture}
        = \frac{1}{6}
        \left(\,
            \begin{tikzpicture}[anchorbase]
                \draw (0,-0.4) -- (0,-0.3) arc(180:0:0.2) -- (0.4,-0.4);
                \draw (-0.4,-0.4) \braidup (0,0.3);
            \end{tikzpicture}
            +
            \begin{tikzpicture}[anchorbase]
                \draw (0,-0.4) -- (0,-0.3) arc(0:180:0.2) -- (-0.4,-0.4);
                \draw (0.4,-0.4) \braidup (0,0.3);
            \end{tikzpicture}
            +
            \begin{tikzpicture}[centerzero]
                \draw (-0.2,-0.3) -- (-0.2,-0.1) arc(180:0:0.2) -- (0.2,-0.3);
                \draw (-0.3,0.3) \braiddown (0,-0.3);
            \end{tikzpicture}
            \,
        \right).
    \]
    The fact that $\Phi$ respects this relation follows from \cref{mango} and the fact that the operation $\circ$ on $A$ is commutative.  Hence, we see that \cref{magic} essentially corresponds to the Cayley--Hamilton theorem for the Albert algebra; see the proof of \cref{raspberry}.
\end{rem}

Since the category $\fg$-mod is idempotent complete, we have an induced functor
\begin{equation} \label{lake}
    \Kar(\Phi) \colon \Kar \left( \Fcat_{7/3,26} \right) \to \fg\md.
\end{equation}

\begin{prop} \label{splat}
    The functor $\Kar(\Phi)$ of \cref{lake} is full and essentially surjective.
\end{prop}

\begin{proof}
    Fullness follows from \cref{FunctorFull}, so it remains to show that $\Kar(\Phi)$ is essentially surjective.  If $\lambda = \sum_{i=1}^4 \lambda_i \omega_i$, with $\lambda_i \in \Z_{\ge 0}$, then $V_\lambda$ is the submodule of $\bigotimes_{i=1}^4 V_{\omega_i}^{\otimes \lambda_i}$ generated by the one-dimensional $\lambda$ weight space.  Since the category $\fg$-mod is semisimple, this implies that $V_\lambda$ is a direct summand of $\bigotimes_{i=1}^4 V_{\omega_i}^{\otimes \lambda_i}$.  Therefore, it suffices to show that the image of $\Kar(\Phi)$ contains the fundamental representations $V_{\omega_i}$ for $i \in \{1,2,3,4\}$.  We see from \cref{2decomp} that $V_{\omega_1}$ and $V_{\omega_3}$ are contained in $V_{\omega_4}^{\otimes 2}$.  It also follows from \cite[Ch.~11,~Table~7]{MPR90} that $V_{\omega_2}$ is contained in $V_{\omega_3} \otimes V_{\omega_4}$.
\end{proof}

\begin{cor} \label{SW}
    We have a surjective algebra homomorphism
    \[
        \Fcat_{7/3,26}(\go^{\otimes k}, \go^{\otimes k})
        \twoheadrightarrow \End_\fg(V^{\otimes k}),\quad k \in \N.
    \]
\end{cor}

\Cref{SW} implies that the endomorphism algebras of $\Fcat$ play the role in type $F_4$ that the group algebra of the symmetric group (or the oriented Brauer algebras if one includes the dual of the natural module) plays in type $A$ and that the Brauer algebras play in types $BCD$.

We conclude this section with some conjectures.

\begin{conj} \label{faithful}
    The functor $\Kar(\Phi)$ is faithful, and hence is an equivalence of categories.
\end{conj}

\begin{rem} \label{hesitant}
    A weaker form of \cref{faithful} would be that the kernel of $\Kar(\Phi)$ consists of negligible morphisms.  Then the functor $\Kar(\Phi)$ would induce an equivalence of categories between $\fg$-mod and the semisimplification of $\Fcat_{7/3,26}$, which is the quotient of $\Fcat_{7/3,26}$ by the tensor ideal of negligible morphisms.
\end{rem}

A string diagram built from the generating morphisms \cref{lego} via tensor product and composition can be viewed as a graph.  Here we view $\mergemor$ as a trivalent vertex and $\crossmor$ as two edges (that is, we do \emph{not} view the crossing as a vertex).  We say such a graph is \emph{component-planar} if its connected components are planar graphs.  For example, \cref{bigfive} is a complete list of the $5$ acyclic component-planar graphs $\go^{\otimes 2} \to \go^{\otimes 2}$, and \cref{brutal} is a complete list of the $15$ acyclic component-planar graphs $\go^{\otimes 2} \to \go^{\otimes 3}$.

\begin{conj} \label{pipe}
    For $m,n \in \N$, the component-planar graphs whose cycles are all of length at least six span $\Fcat(\go^{\otimes m}, \go^{\otimes n})$.
\end{conj}

An earlier version of this paper contained a stronger version of \cref{pipe}, conjecturing that the given graphs form a \emph{basis} of the endomorphism spaces.  However, it was pointed out to us by B.~Westbury that computer calculations show that, for $m+n \ge 8$, the number of such graphs is larger than the dimension of the corresponding endomorphism space.

\section{Fundamental modules\label{sec:fundamental}}

Our goal in this final section is to describe the objects in $\Kar(\Fcat)$ sent, under the functor $\Kar(\Phi)$ of \cref{lake}, to the four fundamental $\fg$-modules.  Some of our intermediate results will be valid for more general $\alpha$ and $\delta$.  However, throughout this section we suppose that
\[
    \delta \notin \{-10,-2,0\}.
\]
We continue to work over the field $\kk=\C$.

Recall the definition of the antisymmetrizers in \cref{boxes}.  Let
\begin{gather*}
    e_0 = \frac{1}{\delta}\, \hourglass,\qquad
    e_1 =  \frac{8}{\delta+10}
    \left(
        \begin{tikzpicture}[centerzero]
            \draw (-0.2,-0.5) -- (-0.2,0.5);
            \draw (0.2,-0.5) -- (0.2,0.5);
            \antbox{-0.3,-0.1}{0.3,0.1};
        \end{tikzpicture}
        + \frac{\delta+2}{4\alpha}\,
        \begin{tikzpicture}[anchorbase]
            \draw (-0.2,-0.4) -- (-0.2,0.6);
            \draw (0.2,-0.4) -- (0.2,0.6);
            \draw (-0.2,0.35) -- (0.2,0.35);
            \antbox{-0.3,-0.1}{0.3,0.1};
        \end{tikzpicture}
    \right)
    ,\\
    e_3 = \frac{\delta+2}{\delta+10}
    \left(
        \begin{tikzpicture}[centerzero]
            \draw (-0.2,-0.5) -- (-0.2,0.5);
            \draw (0.2,-0.5) -- (0.2,0.5);
            \antbox{-0.3,-0.1}{0.3,0.1};
        \end{tikzpicture}
        - \frac{2}{\alpha}\,
        \begin{tikzpicture}[anchorbase]
            \draw (-0.2,-0.4) -- (-0.2,0.6);
            \draw (0.2,-0.4) -- (0.2,0.6);
            \draw (-0.2,0.35) -- (0.2,0.35);
            \antbox{-0.3,-0.1}{0.3,0.1};
        \end{tikzpicture}
    \right)
    ,\qquad
    e_4 = \frac{1}{\alpha} \Imor
    ,\qquad
    \tilde{e} =
    \begin{tikzpicture}[centerzero]
        \draw (-0.15,-0.35) -- (-0.15,0.35);
        \draw (0.15,-0.35) -- (0.15,0.35);
        \symbox{-0.25,-0.1}{0.25,0.1};
    \end{tikzpicture}
    - \frac{1}{\delta}\, \hourglass - \frac{1}{\alpha}\, \Imor.
\end{gather*}

\begin{lem} \label{sponge}
    We have $ef,fe \in \kk e$ for all
    \[
        f \in \left\{ \jail\, ,\ \hourglass,\ \crossmor,\ \Imor,\ \Hmor \right\},
        \qquad
        e \in \{e_0,e_1,e_3,e_4,\tilde{e}\}.
    \]
\end{lem}

\begin{proof}
    Since the given choices of $e$ and $f$ are invariant under rotation by $180\degree$ (for $e$, this follows from \cref{ladderslip}), it suffices to show that $ef \in \kk e$.  This is a straightforward computation.  For example,
    \[
        \jail \circ e_1 = e_1,\quad
        \hourglass \circ e_1 \overset{\cref{pomegranate}}{\underset{\cref{ladderslip}}{=}} 0,\quad
        \crossmor \circ e_1 \overset{\cref{pomegranate}}{\underset{\cref{ladderslip}}{=}} - e_1,\quad
        \Imor \circ e_1 \overset{\cref{pomegranate}}{\underset{\cref{ladderslip}}{=}} 0,\quad
        \Hmor \circ e_1 \overset{\cref{sqburst}}{\underset{\cref{pomegranate}}{=}} \frac{\alpha}{2} e_1
    \]
    and
    \begin{gather*}
        \jail \circ \tilde{e} = \tilde{e},\quad
        \hourglass \circ \tilde{e} \overset{\cref{pomegranate}}{\underset{\cref{chess}}{=}} 0,\quad
        \crossmor \circ \tilde{e} \overset{\cref{pomegranate}}{\underset{\cref{chess}}{=}} \tilde{e},\quad
        \Imor \circ \tilde{e} \overset{\cref{pomegranate}}{\underset{\cref{chess}}{=}} 0,
        \\
        \Hmor \circ \tilde{e} \overset{\cref{chess}}{\underset{\cref{triangle}}{=}}
        \frac{1}{2} \left( \Hmor + \dotcross \right)
        - \frac{\alpha}{\delta}\, \hourglass
        + \frac{\delta-2}{2(\delta+2)}\, \Imor
        \overset{\cref{magic}}{=} \frac{2\alpha}{\delta+2} \tilde{e}.
    \end{gather*}
    The computations for $e \in \{e_0,e_3,e_4\}$ are similar.
\end{proof}

\begin{lem} \label{neanderthal}
    We have a decomposition
    \begin{equation} \label{coals}
        1_{\go \otimes \go} = e_0 + e_1 + e_3 + e_4 + \tilde{e}
    \end{equation}
    of $1_{\go \otimes \go}$ as a sum of pairwise orthogonal idempotents.
\end{lem}

\begin{proof}
    A straightforward computation verifies \cref{coals}; one first uses that $\jail$ is the sum of the symmetrizer and the antisymmetrizer, and then decomposes further.  Next we show that $e_0,e_1,e_3,e_4$ are idempotent.  We have $e_0^2=e_0$ and $e_4^2=e_4$ by the fourth and third relations in \cref{chess}, respectively.  Next,
    \begin{multline*}
        e_1^2
        \overset{\cref{ladderslip}}{=} \frac{64}{(\delta+10)^2}
        \left(
            \begin{tikzpicture}[centerzero]
                \draw (-0.2,-0.5) -- (-0.2,0.5);
                \draw (0.2,-0.5) -- (0.2,0.5);
                \antbox{-0.3,-0.1}{0.3,0.1};
            \end{tikzpicture}
            + \frac{\delta+2}{4\alpha}\,
            \begin{tikzpicture}[anchorbase]
                \draw (-0.2,-0.4) -- (-0.2,0.6);
                \draw (0.2,-0.4) -- (0.2,0.6);
                \draw (-0.2,0.35) -- (0.2,0.35);
                \antbox{-0.3,-0.1}{0.3,0.1};
            \end{tikzpicture}
            + \frac{(\delta+2)^2}{16 \alpha^2}
            \begin{tikzpicture}[anchorbase]
                \draw (-0.2,-0.4) -- (-0.2,0.6);
                \draw (0.2,-0.4) -- (0.2,0.6);
                \draw (-0.2,0.2) -- (0.2,0.2);
                \draw (-0.2,0.4) -- (0.2,0.4);
                \antbox{-0.3,-0.2}{0.3,0};
            \end{tikzpicture}
        \right)
        \\
        \overset{\mathclap{\cref{sqburst}}}{\underset{\cref{pomegranate}}{=}}\
        \frac{64}{(\delta+10)^2}
        \left(
            \frac{\delta+10}{8}
            \begin{tikzpicture}[centerzero]
                \draw (-0.2,-0.5) -- (-0.2,0.5);
                \draw (0.2,-0.5) -- (0.2,0.5);
                \antbox{-0.3,-0.1}{0.3,0.1};
            \end{tikzpicture}
            + \frac{(\delta+10)(\delta+2)}{32\alpha}\,
            \begin{tikzpicture}[anchorbase]
                \draw (-0.2,-0.4) -- (-0.2,0.6);
                \draw (0.2,-0.4) -- (0.2,0.6);
                \draw (-0.2,0.35) -- (0.2,0.35);
                \antbox{-0.3,-0.1}{0.3,0.1};
            \end{tikzpicture}
        \right)
        = e_1.
    \end{multline*}
    Since
    \[
        e_3 =
        \begin{tikzpicture}[centerzero]
            \draw (-0.15,-0.35) -- (-0.15,0.35);
            \draw (0.15,-0.35) -- (0.15,0.35);
            \antbox{-0.25,-0.1}{0.25,0.1};
        \end{tikzpicture}
        \, - e_1,
    \]
    it then follows that $e_3^2=e_3$.

    It remains to show that the given idempotents are orthogonal.  This is a mostly routine verification using \cref{pomegranate} and the fact that the symmetrizer and antisymmetrizer are orthogonal.  The most involved computation is the verification that $e_1$ and $e_3$ are orthogonal.  We see this by computing
    \[
        e_1 e_3
        = \frac{8(\delta+2)}{(\delta+10)^2}
        \left(
            \begin{tikzpicture}[centerzero]
                \draw (-0.2,-0.5) -- (-0.2,0.5);
                \draw (0.2,-0.5) -- (0.2,0.5);
                \antbox{-0.3,-0.1}{0.3,0.1};
            \end{tikzpicture}
            + \frac{\delta-6}{4\alpha}\,
            \begin{tikzpicture}[anchorbase]
                \draw (-0.2,-0.4) -- (-0.2,0.6);
                \draw (0.2,-0.4) -- (0.2,0.6);
                \draw (-0.2,0.35) -- (0.2,0.35);
                \antbox{-0.3,-0.1}{0.3,0.1};
            \end{tikzpicture}
            - \frac{\delta+2}{2 \alpha^2}
            \begin{tikzpicture}[anchorbase]
                \draw (-0.2,-0.4) -- (-0.2,0.6);
                \draw (0.2,-0.4) -- (0.2,0.6);
                \draw (-0.2,0.2) -- (0.2,0.2);
                \draw (-0.2,0.4) -- (0.2,0.4);
                \antbox{-0.3,-0.2}{0.3,0};
            \end{tikzpicture}
        \right)
        \overset{\cref{sqburst}}{\underset{\cref{pomegranate}}{=}}
        0.
    \]
    It then follows from \cref{ladderslip} that $e_3 e_1 = e_1 e_3 = 0$.
\end{proof}

\begin{rem}
    If we knew that the morphisms \cref{bigfive} spanned $\Fcat(\go^{\otimes 2}, \go^{\otimes 2})$, then it would follow from \cref{sponge} that
    \[
        \dim_\kk \left( e \Fcat(\go^{\otimes 2},\go^{\otimes 2}) e \right) = 1
        \quad \text{for all } e \in \{e_0,e_1,e_3,e_4,\tilde{e}\},
    \]
    and hence that $e_0,e_1,e_3,e_4,\tilde{e}$ are primitive.  For example, this would be the case if \cref{pipe} holds.
\end{rem}

Recall from \cref{sec:functor} our conventions for labeling of weights of $\fg$ and that $V = V_{\omega_4}$.

\begin{theo} \label{hacky}
    Suppose $\alpha = \frac{7}{3}$ and $\delta = 26$.  Then
    \[
        \Kar(\Phi)(\go^{\otimes 2}, e_1) = V_{\omega_1},\quad
        \Kar(\Phi)(\go^{\otimes 2}, e_3) = V_{\omega_3},\quad
        \Kar(\Phi)(\go, 1_\go) = V_{\omega_4}, \quad
        \Kar(\Phi)(\go^{\otimes 2},\tilde{e}) = V_{2\omega_4}.
    \]
\end{theo}

\begin{proof}
    In $\fg$-mod, we have the tensor product decomposition \cref{2decomp} and dimensions
    \[
        \dim V_0 = 1,\quad
        \dim V_{\omega_1} = 52,\quad
        \dim V_{\omega_3} = 273,\quad
        \dim V_{\omega_4} = 26,\quad
        \dim V_{2 \omega_4} = 324.
    \]
    See, for example, \cite[Ch.~6,~Table~2]{MPR90}.  By \cref{neanderthal}, we also have the decomposition
    \[
        V^{\otimes 2} \cong \Kar(\Phi)(\go^{\otimes 2},e_0)
        \oplus \Kar(\Phi)(\go^{\otimes 2}, e_1)
        \oplus \Kar(\Phi)(\go^{\otimes 2}, e_3)
        \oplus \Kar(\Phi)(\go^{\otimes 2}, e_4)
        \oplus \Kar(\Phi)(\go^{\otimes 2}, \tilde{e}).
    \]
    Thus, if we can show that the images of the given elements of $\Kar(\Fcat)$ have the expected dimensions, the theorem follows.  Since the dimension of a $\fg$-module is the trace of its identity endomorphism, the dimension of $\Kar(\Phi)(\go^{\otimes n}, e)$ is the image under $\Phi$ of
    \[
        \begin{tikzpicture}[centerzero]
            \draw (-0.7,-0.2) rectangle (0,0.2);
            \node at (-0.35,0) {$e$};
            \draw (-0.2,0.2) arc(180:0:0.2) -- (0.2,-0.2) arc(360:180:0.2);
            \draw (-0.4,0.2) arc(180:0:0.4) -- (0.4,-0.2) arc(360:180:0.4);
            \draw (-0.6,0.2) arc(180:0:0.6) -- (0.6,-0.2) arc(360:180:0.6);
        \end{tikzpicture}
        \in \Fcat(\go,\go),
    \]
    where there are $n$ strands. (Here, and in what follows, we identify $\lambda 1_\one$ with $\lambda \in \kk$.)

    Note that
    \begin{gather*}
        \begin{tikzpicture}[centerzero]
            \draw (-0.2,0.1) arc(180:0:0.2) -- (0.2,-0.1) arc(360:180:0.2);
            \draw (-0.4,0.1) arc(180:0:0.4) -- (0.4,-0.1) arc(360:180:0.4);
            \antbox{-0.5,-0.1}{-0.1,0.1};
        \end{tikzpicture}
        =
        \frac{1}{2}\,
        \begin{tikzpicture}[centerzero]
            \draw (-0.2,-0.1) -- (-0.2,0.1) arc(180:0:0.2) -- (0.2,-0.1) arc(360:180:0.2);
            \draw (-0.4,-0.1) -- (-0.4,0.1) arc(180:0:0.4) -- (0.4,-0.1) arc(360:180:0.4);
        \end{tikzpicture}
        - \frac{1}{2}\,
        \begin{tikzpicture}[centerzero]
            \draw (-0.4,-0.15) \braidup (-0.2,0.15) arc(180:0:0.15) -- (0.1,-0.15) arc(360:180:0.15);
            \draw (-0.2,-0.15) \braidup (-0.4,0.15) arc(180:0:0.35) -- (0.3,-0.15) arc(360:180:0.35);
        \end{tikzpicture}
        \overset{\cref{chess}}{=}
        \frac{\delta(\delta-1)}{2},
        \\
        \begin{tikzpicture}[anchorbase]
            \draw (-0.2,0.1) -- (-0.2,0.3) arc(180:0:0.2) -- (0.2,-0.1) arc(360:180:0.2);
            \draw (-0.4,0.1) -- (-0.4,0.3) arc(180:0:0.4) -- (0.4,-0.1) arc(360:180:0.4);
            \draw (-0.4,0.25) -- (-0.2,0.25);
            \antbox{-0.5,-0.1}{-0.1,0.1};
        \end{tikzpicture}
        =
        \frac{1}{2}\,
        \begin{tikzpicture}[centerzero]
            \draw (0,0) circle(0.2);
            \draw (0,0) circle(0.4);
            \draw (-0.4,0) -- (-0.2,0);
        \end{tikzpicture}
        -\frac{1}{2}\!\!\!
        \begin{tikzpicture}[centerzero]
            \draw (0,0) to[out=60,in=90,looseness=1.5] (0.4,0) to[out=-90,in=-60,looseness=1.5] (0,0);
            \draw (0,0) to[out=135,in=90,looseness=2.5] (0.6,0) to[out=-90,in=-135,looseness=2.5] (0,0);
            \opendot{0,0};
        \end{tikzpicture}
        \overset{\cref{chess}}{=} 0 - \frac{1}{2}\!\!\!
        \begin{tikzpicture}[centerzero]
            \draw (0,-0.1) -- (0,0.1) -- (0.1,0.2) to[out=45,in=90] (0.4,0) to[out=-90,in=-45] (0.1,-0.2) -- (0,-0.1);
            \draw (0,0.1) -- (-0.1,0.2) to[out=135,in=90,looseness=2] (0.6,0) to[out=-90,in=-135,looseness=2] (-0.1,-0.2) -- (0,-0.1);
        \end{tikzpicture}
        \overset{\cref{chess}}{=} - \frac{\alpha \delta}{2}.
    \end{gather*}
    Therefore, we have
    \begin{gather*}
        \dim (\go^{\otimes 2}, e_1)
        = \frac{8}{\delta+10} \left( \frac{\delta(\delta-1)}{2} - \frac{\delta (\delta+2)}{8} \right)
        = \frac{3 \delta (\delta-2)}{\delta+10}
        = 52,
        \\
        \dim (\go^{\otimes 2}, e_3)
        = \frac{\delta+2}{\delta+10} \left( \frac{\delta(\delta-1)}{2} + \delta \right)
        = \frac{\delta(\delta+1)(\delta+2)}{2(\delta+10)}
        = 273.
    \end{gather*}
    It is also clear that $\dim(\go,1_\one) = \dim(\go^{\otimes 2}, e_4) = \delta = 26$ and that $\dim(\go^{\otimes 2},e_0) = 1$.  It then follows that $\dim(\go^{\otimes 2},\tilde{e}) = 26^2 - 1 - 52 - 273 - 26 = 324$.
\end{proof}

Note that \cref{hacky} does not describe an object in $\Kar(\Fcat)$ that is mapped to the second fundamental representation $V_{\omega_2}$.  In $\fg$-mod, the lowest tensor power of $V$ in which $V_{\omega_2}$ appears is $V^{\otimes 3}$.  Thus, we expect that there exists an idempotent $e_2 \in \Fcat(\go^{\otimes 3}, \go^{\otimes 3})$ such that $\Kar(\Phi)(\go^{\otimes 3},e_2) = V_{\omega_2}$.  (Such an idempotent is guaranteed to exist if \cref{faithful} holds.)  In fact, we can say a bit more.  A computation similar to the ones in the proof of \cref{hacky} shows that
\begin{equation}
    \begin{tikzpicture}[centerzero]
        \antbox{-0.7,-0.1}{-0.1,0.1};
        \draw (-0.2,0.1) arc(180:0:0.2) -- (0.2,-0.1) arc(360:180:0.2);
        \draw (-0.4,0.1) arc(180:0:0.4) -- (0.4,-0.1) arc(360:180:0.4);
        \draw (-0.6,0.1) arc(180:0:0.6) -- (0.6,-0.1) arc(360:180:0.6);
    \end{tikzpicture}
    = \frac{\delta(\delta-1)(\delta-2)}{6}
    = 2600 = 1274 + 273 + 1053
    \quad \text{when } \delta = 26.
\end{equation}
The tensor product decomposition of $V^{\otimes 3}$ is given in \cref{3decomp}.  Since
\begin{gather*}
    \dim V_0 = 1,\
    \dim V_{\omega_1} = 52,\
    \dim V_{\omega_2} = 1274,\
    \dim V_{\omega_3} = 273,\
    \dim V_{\omega_4} = 26,\\
    \dim V_{\omega_1+\omega_4} = 1053,\
    \dim V_{\omega_3+\omega_4} = 4096,\
    \dim V_{2\omega_4} = 324, \text{ and }
    \dim V_{3\omega_4} = 2652,
\end{gather*}
(see, for example, \cite[Ch.~6,~Table~2]{MPR90}) the only submodule of $V^{\otimes 3}$ of dimension $2600$ is $V_{\omega_2} \oplus V_{\omega_3} \oplus V_{\omega_1+\omega_4}$.  Thus, we expect that the antisymmetrizer
$
    \begin{tikzpicture}[centerzero]
        \draw (-0.2,-0.3) -- (-0.2,0.3);
        \draw (0,-0.3) -- (0,0.3);
        \draw (0.2,-0.3) -- (0.2,0.3);
        \antbox{-0.3,-0.1}{0.3,0.1};
    \end{tikzpicture}
$
decomposes as a sum of three orthogonal minimal idempotents, one of which is $e_2$.  (This is guaranteed to happen if \cref{faithful} holds.)  Unfortunately, the computations required to find this idempotent explicitly are unwieldy.

\begin{lem}
    We have
    \[
        \begin{tikzpicture}[centerzero]
            \draw (-0.6,-0.2) rectangle (-0.1,0.2);
            \node at (-0.35,0) {$e_1$};
            \draw (0.6,-0.2) rectangle (0.1,0.2);
            \node at (0.35,0) {$e_1$};
            \draw (-0.25,0.2) arc(180:0:0.25);
            \draw (0.25,-0.2) arc(360:180:0.25);
            \draw (-0.45,0.2) to[out=up,in=225] (0,0.7) -- (0,0.9);
            \draw (0.45,0.2) to[out=up,in=-45] (0,0.7);
            \draw (-0.45,-0.2) to[out=down,in=135] (0,-0.7) -- (0,-0.9);
            \draw (0.45,-0.2) to[out=down,in=45] (0,-0.7);
        \end{tikzpicture}
        =
        \frac{3 \alpha (2-\delta) (\delta - 26)}{4 (\delta + 10)^2} 1_\go.
    \]
\end{lem}

\begin{proof}
    We compute
    \begin{gather*}
        \begin{tikzpicture}[centerzero]
            \antbox{-0.6,-0.1}{-0.1,0.1};
            \antbox{0.6,-0.1}{0.1,0.1};
            \draw (-0.25,0.1) arc(180:0:0.25);
            \draw (0.25,-0.1) arc(360:180:0.25);
            \draw (-0.45,0.1) to[out=up,in=225] (0,0.6) -- (0,0.8);
            \draw (0.45,0.1) to[out=up,in=-45] (0,0.6);
            \draw (-0.45,-0.1) to[out=down,in=135] (0,-0.6) -- (0,-0.8);
            \draw (0.45,-0.1) to[out=down,in=45] (0,-0.6);
        \end{tikzpicture}
        =
        \frac{1}{4}
        \left(
            \begin{tikzpicture}[centerzero]
                \draw (-0.25,-0.1) -- (-0.25,0.1) arc(180:0:0.25) -- (0.25,-0.1) arc(360:180:0.25) -- cycle;
                \draw (0,-0.8) -- (0,-0.6) to[out=135,in=down] (-0.45,-0.1) -- (-0.45,0.1) to[out=up,in=225] (0,0.6) -- (0,0.8);
                \draw (0,-0.6) to[out=45,in=down] (0.45,-0.1) -- (0.45,0.1) to[out=up,in=-45] (0,0.6) -- (0,0.8);
            \end{tikzpicture}
            -
            \begin{tikzpicture}[centerzero]
                \draw (0,-0.8) -- (0,-0.6) to[out=135,in=down] (-0.45,-0.2) to[out=up,in=180] (0,0.25) arc(90:-90:0.25) to[out=180,in=down] (-0.45,0.2) to[out=up,in=225] (0,0.6) -- (0,0.8);
                \draw (0,-0.6) to[out=45,in=down] (0.45,-0.1) -- (0.45,0.1) to[out=up,in=-45] (0,0.6) -- (0,0.8);
            \end{tikzpicture}
            -
            \begin{tikzpicture}[centerzero,xscale=-1]
                \draw (0,-0.8) -- (0,-0.6) to[out=135,in=down] (-0.45,-0.2) to[out=up,in=180] (0,0.25) arc(90:-90:0.25) to[out=180,in=down] (-0.45,0.2) to[out=up,in=225] (0,0.6) -- (0,0.8);
                \draw (0,-0.6) to[out=45,in=down] (0.45,-0.1) -- (0.45,0.1) to[out=up,in=-45] (0,0.6) -- (0,0.8);
            \end{tikzpicture}
            +
            \begin{tikzpicture}[centerzero]
                \draw (0,-0.8) -- (0,-0.5) to[out=135,in=down] (-0.4,-0.2) to[out=up,in=180] (0,0.2) to[out=0,in=up] (0.4,-0.2) to[out=down,in=45] (0,-0.5);
                \draw (0,0.8) -- (0,0.5) to[out=-135,in=up] (-0.4,0.2) to[out=down,in=180] (0,-0.2) to[out=0,in=down] (0.4,0.2) to[out=up,in=-45] (0,0.5);
            \end{tikzpicture}
        \right)
        \overset{\cref{chess}}{=}
        \frac{\alpha(\delta-2)}{4} 1_\go,
        \\
        \begin{tikzpicture}[centerzero]
            \draw (-0.25,-0.1) -- (-0.25,0.1) arc(180:0:0.25) -- (0.25,-0.1) arc(360:180:0.25) -- cycle;
            \draw (0,-0.8) -- (0,-0.6) to[out=135,in=down] (-0.45,-0.1) -- (-0.45,0.1) to[out=up,in=225] (0,0.6) -- (0,0.8);
            \draw (0,-0.6) to[out=45,in=down] (0.45,-0.1) -- (0.45,0.1) to[out=up,in=-45] (0,0.6) -- (0,0.8);
            \antbox{-0.6,-0.1}{-0.1,0.1};
            \antbox{0.1,-0.25}{0.6,-0.05};
            \draw (0.25,0.1) -- (0.45,0.1);
        \end{tikzpicture}
        =
        \frac{1}{4}
        \left(
            \begin{tikzpicture}[centerzero]
                \draw (-0.25,-0.1) -- (-0.25,0.1) arc(180:0:0.25) -- (0.25,-0.1) arc(360:180:0.25) -- cycle;
                \draw (0,-0.8) -- (0,-0.6) to[out=135,in=down] (-0.45,-0.1) -- (-0.45,0.1) to[out=up,in=225] (0,0.6) -- (0,0.8);
                \draw (0,-0.6) to[out=45,in=down] (0.45,-0.1) -- (0.45,0.1) to[out=up,in=-45] (0,0.6) -- (0,0.8);
                \draw (0.25,0) -- (0.45,0);
            \end{tikzpicture}
            -
            \begin{tikzpicture}[centerzero]
                \draw (0,-0.8) -- (0,-0.6) to[out=135,in=down] (-0.45,-0.2) to[out=up,in=180] (0,0.25) arc(90:-90:0.25) to[out=180,in=down] (-0.45,0.2) to[out=up,in=225] (0,0.6) -- (0,0.8);
                \draw (0,-0.6) to[out=45,in=down] (0.45,-0.1) -- (0.45,0.1) to[out=up,in=-45] (0,0.6) -- (0,0.8);
                \draw (0.25,0) -- (0.45,0);
            \end{tikzpicture}
            -
            \begin{tikzpicture}[centerzero,xscale=-1]
                \draw (0,-0.8) -- (0,-0.6) to[out=135,in=down] (-0.45,-0.2) to[out=up,in=180] (0,0.25) arc(90:-90:0.25) to[out=180,in=down] (-0.45,0.2) to[out=up,in=225] (0,0.6) -- (0,0.8);
                \draw (0,-0.6) to[out=45,in=down] (0.45,-0.1) -- (0.45,0.1) to[out=up,in=-45] (0,0.6) -- (0,0.8);
                \opendot{-0.4,0};
            \end{tikzpicture}
            +
            \begin{tikzpicture}[centerzero]
                \draw (0,-0.8) -- (0,-0.5) to[out=135,in=down] (-0.4,-0.2) to[out=up,in=180] (0,0.2) to[out=0,in=up] (0.4,-0.2) to[out=down,in=45] (0,-0.5);
                \draw (0,0.8) -- (0,0.5) to[out=-135,in=up] (-0.4,0.2) to[out=down,in=180] (0,-0.2) to[out=0,in=down] (0.4,0.2) to[out=up,in=-45] (0,0.5);
                \opendot{0.34,0};
            \end{tikzpicture}
        \right)
        \overset{\cref{chess}}{=}
        \frac{1}{4}
        \left(
            0 -
            \begin{tikzpicture}[centerzero]
                \draw (0,-0.4) -- (0,-0.2) to[out=135,in=down] (-0.15,0) to[out=up,in=225] (0,0.2) -- (0,0.4);
                \draw (0,-0.2) to[out=45,in=down] (0.15,0) to[out=up,in=-45] (0,0.2);
                \draw (-0.15,0) -- (0.15,0);
            \end{tikzpicture}
            -
            \begin{tikzpicture}[centerzero]
                \draw (0,-0.7) -- (0,-0.5) to[out=45,in=-45] (0.2,-0.15) -- (0.2,0.15) to[out=45,in=-45] (0,0.5) -- (0,0.7);
                \draw (0.2,0.15) to[out=135,in=up,looseness=1.5] (0,0) to[out=down,in=225,looseness=1.5] (0.2,-0.15);
                \draw (0,-0.5) to[out=135,in=down] (-0.2,0) to[out=up,in=225] (0,0.5);
            \end{tikzpicture}
            +
            \begin{tikzpicture}[centerzero]
                \draw (0,-0.7) -- (0,-0.5) to[out=45,in=-45,looseness=1.5] (0,-0.1) -- (0,0.1) to[out=45,in=-45,looseness=1.5] (0,0.5) -- (0,0.7);
                \draw (0,-0.5) to[out=135,in=-135,looseness=1.5] (0,-0.1) -- (0,0.1) to[out=135,in=-135,looseness=1.5] (0,0.5);
            \end{tikzpicture}
        \right)
        \overset{\cref{chess}}{\underset{\cref{triangle}}{=}}
        \frac{\alpha^2(\delta-2)}{8(\delta+2)} 1_\go,
        \\
        \begin{tikzpicture}[centerzero]
            \draw (-0.25,-0.1) -- (-0.25,0.1) arc(180:0:0.25) -- (0.25,-0.1) arc(360:180:0.25) -- cycle;
            \draw (0,-0.8) -- (0,-0.6) to[out=135,in=down] (-0.45,-0.1) -- (-0.45,0.1) to[out=up,in=225] (0,0.6) -- (0,0.8);
            \draw (0,-0.6) to[out=45,in=down] (0.45,-0.1) -- (0.45,0.1) to[out=up,in=-45] (0,0.6) -- (0,0.8);
            \antbox{-0.6,-0.25}{-0.1,-0.05};
            \antbox{0.1,-0.25}{0.6,-0.05};
            \draw (0.25,0.1) -- (0.45,0.1);
            \draw (-0.25,0.1) -- (-0.45,0.1);
        \end{tikzpicture}
        =
        \frac{1}{4}
        \left(
            \begin{tikzpicture}[centerzero]
                \draw (-0.25,-0.1) -- (-0.25,0.1) arc(180:0:0.25) -- (0.25,-0.1) arc(360:180:0.25) -- cycle;
                \draw (0,-0.8) -- (0,-0.6) to[out=135,in=down] (-0.45,-0.1) -- (-0.45,0.1) to[out=up,in=225] (0,0.6) -- (0,0.8);
                \draw (0,-0.6) to[out=45,in=down] (0.45,-0.1) -- (0.45,0.1) to[out=up,in=-45] (0,0.6) -- (0,0.8);
                \draw (0.25,0) -- (0.45,0);
                \draw (-0.25,0) -- (-0.45,0);
            \end{tikzpicture}
            -
            \begin{tikzpicture}[centerzero]
                \draw (0,-0.8) -- (0,-0.6) to[out=135,in=down] (-0.45,-0.2) to[out=up,in=180] (0,0.25) arc(90:-90:0.25) to[out=180,in=down] (-0.45,0.2) to[out=up,in=225] (0,0.6) -- (0,0.8);
                \draw (0,-0.6) to[out=45,in=down] (0.45,-0.1) -- (0.45,0.1) to[out=up,in=-45] (0,0.6) -- (0,0.8);
                \draw (0.25,0) -- (0.45,0);
                \opendot{-0.4,0};
            \end{tikzpicture}
            -
            \begin{tikzpicture}[centerzero,xscale=-1]
                \draw (0,-0.8) -- (0,-0.6) to[out=135,in=down] (-0.45,-0.2) to[out=up,in=180] (0,0.25) arc(90:-90:0.25) to[out=180,in=down] (-0.45,0.2) to[out=up,in=225] (0,0.6) -- (0,0.8);
                \draw (0,-0.6) to[out=45,in=down] (0.45,-0.1) -- (0.45,0.1) to[out=up,in=-45] (0,0.6) -- (0,0.8);
                \opendot{-0.4,0};
                \draw (0.25,0) -- (0.45,0);
            \end{tikzpicture}
            +
            \begin{tikzpicture}[centerzero]
                \draw (0,-0.8) -- (0,-0.5) to[out=135,in=down] (-0.4,-0.2) to[out=up,in=180] (0,0.2) to[out=0,in=up] (0.4,-0.2) to[out=down,in=45] (0,-0.5);
                \draw (0,0.8) -- (0,0.5) to[out=-135,in=up] (-0.4,0.2) to[out=down,in=180] (0,-0.2) to[out=0,in=down] (0.4,0.2) to[out=up,in=-45] (0,0.5);
                \opendot{0.34,0};
                \opendot{-0.34,0};
            \end{tikzpicture}
        \right)
        \overset{\cref{chess}}{=}
        \frac{1}{4}
        \left(
            \alpha\,
            \begin{tikzpicture}[centerzero]
                \draw (0,-0.4) -- (0,-0.2) to[out=135,in=down] (-0.15,0) to[out=up,in=225] (0,0.2) -- (0,0.4);
                \draw (0,-0.2) to[out=45,in=down] (0.15,0) to[out=up,in=-45] (0,0.2);
                \draw (-0.15,0) -- (0.15,0);
            \end{tikzpicture}
            -
            \begin{tikzpicture}[centerzero,xscale=-1]
                \draw (0,-0.7) -- (0,-0.5) to[out=45,in=-45] (0.2,-0.15) -- (0.2,0.15) to[out=45,in=-45] (0,0.5) -- (0,0.7);
                \draw (0.2,0.15) to[out=135,in=up,looseness=1.5] (0,0) to[out=down,in=225,looseness=1.5] (0.2,-0.15);
                \draw (0,-0.5) to[out=135,in=down] (-0.2,0) to[out=up,in=225] (0,0.5);
                \draw (-0.2,0) -- (0,0);
            \end{tikzpicture}
            -
            \begin{tikzpicture}[centerzero]
                \draw (0,-0.7) -- (0,-0.5) to[out=45,in=-45] (0.2,-0.15) -- (0.2,0.15) to[out=45,in=-45] (0,0.5) -- (0,0.7);
                \draw (0.2,0.15) to[out=135,in=up,looseness=1.5] (0,0) to[out=down,in=225,looseness=1.5] (0.2,-0.15);
                \draw (0,-0.5) to[out=135,in=down] (-0.2,0) to[out=up,in=225] (0,0.5);
                \draw (-0.2,0) -- (0,0);
            \end{tikzpicture}
            +
            \begin{tikzpicture}[centerzero]
                \draw (0,-0.8) -- (0,-0.6) to[out=135,in=225] (-0.2,-0.1) -- (-0.2,0.1) to[out=135,in=225] (0,0.6) -- (0,0.8);
                \draw (0,-0.6) to[out=45,in=-45] (0.2,-0.1) -- (0.2,0.1) to[out=45,in=-45] (0,0.6);
                \draw (-0.2,0.1) to[out=45,in=135,looseness=1.5] (0.2,0.1);
                \draw (-0.2,-0.1) to[out=-45,in=-135,looseness=1.5] (0.2,-0.1);
            \end{tikzpicture}
        \right)
        \\ \qquad \qquad \qquad
        \overset{\cref{chess}}{\underset{\cref{triangle}}{=}}
        \frac{\alpha^3 (2 - \delta)(3 \delta + 2)}{16(\delta+2)^2} 1_\go.
    \end{gather*}
    Thus
    \begin{align*}
        \begin{tikzpicture}[centerzero]
            \draw (-0.6,-0.2) rectangle (-0.1,0.2);
            \node at (-0.35,0) {$e_1$};
            \draw (0.6,-0.2) rectangle (0.1,0.2);
            \node at (0.35,0) {$e_1$};
            \draw (-0.25,0.2) arc(180:0:0.25);
            \draw (0.25,-0.2) arc(360:180:0.25);
            \draw (-0.45,0.2) to[out=up,in=225] (0,0.7) -- (0,0.9);
            \draw (0.45,0.2) to[out=up,in=-45] (0,0.7);
            \draw (-0.45,-0.2) to[out=down,in=135] (0,-0.7) -- (0,-0.9);
            \draw (0.45,-0.2) to[out=down,in=45] (0,-0.7);
        \end{tikzpicture}
        &= \frac{64}{(\delta+10)^2}
        \left(
            \begin{tikzpicture}[centerzero]
                \antbox{-0.6,-0.1}{-0.1,0.1};
                \antbox{0.6,-0.1}{0.1,0.1};
                \draw (-0.25,0.1) arc(180:0:0.25);
                \draw (0.25,-0.1) arc(360:180:0.25);
                \draw (-0.45,0.1) to[out=up,in=225] (0,0.6) -- (0,0.8);
                \draw (0.45,0.1) to[out=up,in=-45] (0,0.6);
                \draw (-0.45,-0.1) to[out=down,in=135] (0,-0.6) -- (0,-0.8);
                \draw (0.45,-0.1) to[out=down,in=45] (0,-0.6);
            \end{tikzpicture}
            +
            \frac{\delta+2}{4\alpha}
            \begin{tikzpicture}[centerzero]
                \draw (-0.25,-0.1) -- (-0.25,0.1) arc(180:0:0.25) -- (0.25,-0.1) arc(360:180:0.25) -- cycle;
                \draw (0,-0.8) -- (0,-0.6) to[out=135,in=down] (-0.45,-0.1) -- (-0.45,0.1) to[out=up,in=225] (0,0.6) -- (0,0.8);
                \draw (0,-0.6) to[out=45,in=down] (0.45,-0.1) -- (0.45,0.1) to[out=up,in=-45] (0,0.6) -- (0,0.8);
                \antbox{-0.6,-0.1}{-0.1,0.1};
                \antbox{0.1,-0.25}{0.6,-0.05};
                \draw (0.25,0.1) -- (0.45,0.1);
            \end{tikzpicture}
            +
            \frac{\delta+2}{4\alpha}
            \begin{tikzpicture}[centerzero,xscale=-1]
                \draw (-0.25,-0.1) -- (-0.25,0.1) arc(180:0:0.25) -- (0.25,-0.1) arc(360:180:0.25) -- cycle;
                \draw (0,-0.8) -- (0,-0.6) to[out=135,in=down] (-0.45,-0.1) -- (-0.45,0.1) to[out=up,in=225] (0,0.6) -- (0,0.8);
                \draw (0,-0.6) to[out=45,in=down] (0.45,-0.1) -- (0.45,0.1) to[out=up,in=-45] (0,0.6) -- (0,0.8);
                \antbox{-0.6,-0.1}{-0.1,0.1};
                \antbox{0.1,-0.25}{0.6,-0.05};
                \draw (0.25,0.1) -- (0.45,0.1);
            \end{tikzpicture}
            +
            \frac{(\delta+2)^2}{16 \alpha^2}
            \begin{tikzpicture}[centerzero]
                \draw (-0.25,-0.1) -- (-0.25,0.1) arc(180:0:0.25) -- (0.25,-0.1) arc(360:180:0.25) -- cycle;
                \draw (0,-0.8) -- (0,-0.6) to[out=135,in=down] (-0.45,-0.1) -- (-0.45,0.1) to[out=up,in=225] (0,0.6) -- (0,0.8);
                \draw (0,-0.6) to[out=45,in=down] (0.45,-0.1) -- (0.45,0.1) to[out=up,in=-45] (0,0.6) -- (0,0.8);
                \antbox{-0.6,-0.25}{-0.1,-0.05};
                \antbox{0.1,-0.25}{0.6,-0.05};
                \draw (0.25,0.1) -- (0.45,0.1);
                \draw (-0.25,0.1) -- (-0.45,0.1);
            \end{tikzpicture}
        \right)
        \\
        &=
        \frac{3 \alpha (2-\delta) (\delta - 26)}{4 (\delta + 10)^2} 1_\go.
        \qedhere
    \end{align*}
\end{proof}

\begin{cor} \label{twentysix}
    If $\alpha \ne 0$ and $\delta \notin \{-2,-10,2,26\}$, then
    \[
        \frac{4(\delta+10)^2}{3 \alpha (2-\delta)(\delta-26)}
        \begin{tikzpicture}[centerzero]
            \draw (-0.25,-1.15) -- (-0.25,-1);
            \draw (-0.45,-1.15) -- (-0.45,-1);
            \draw (0.25,-1.15) -- (0.25,-1);
            \draw (0.45,-1.15) -- (0.45,-1);
            \draw (-0.6,-1) rectangle (-0.1,-0.6);
            \node at (-0.35,-0.8) {$e_1$};
            \draw (0.6,-1) rectangle (0.1,-0.6);
            \node at (0.35,-0.8) {$e_1$};
            \draw (-0.25,-0.6) arc(180:0:0.25);
            \draw (-0.45,-0.6) to[out=up,in=225] (0,-0.1) -- (0,0.1) to[out=135,in=down] (-0.45,0.6);
            \draw (0.45,-0.6) to[out=up,in=-45] (0,-0.1);
            \draw (0,0.1) to[out=45,in=down] (0.45,0.6);
            \draw (-0.25,0.6) arc(180:360:0.25);
            \draw (-0.6,1) rectangle (-0.1,0.6);
            \draw (0.6,1) rectangle (0.1,0.6);
            \node at (-0.35,0.8) {$e_1$};
            \node at (0.35,0.8) {$e_1$};
            \draw (-0.25,1.15) -- (-0.25,1);
            \draw (-0.45,1.15) -- (-0.45,1);
            \draw (0.25,1.15) -- (0.25,1);
            \draw (0.45,1.15) -- (0.45,1);
        \end{tikzpicture}
    \]
    is a nonzero idempotent endomorphism.  In particular, $\go$ is a direct summand of $(\go^{\otimes 2}, e_1)^{\otimes 2}$ in $\Kar(\Fcat)$.
\end{cor}

\begin{rem} \label{sack}
    \Cref{twentysix,hacky} explain the importance of the choice $\delta=26$.  The $\fg$-module $V_{\omega_1}$ is the adjoint representation, and $V_{\omega_1}^{\otimes 2}$ does not contain a copy of the natural representation $V = V_{\omega_4}$.  In fact, a lengthy computation shows that the morphism
    
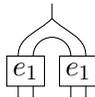
\[
        \begin{tikzpicture}[centerzero]
            \draw (-0.25,-1.15) -- (-0.25,-1);
            \draw (-0.45,-1.15) -- (-0.45,-1);
            \draw (0.25,-1.15) -- (0.25,-1);
            \draw (0.45,-1.15) -- (0.45,-1);
            \draw (-0.6,-1) rectangle (-0.1,-0.6);
            \node at (-0.35,-0.8) {$e_1$};
            \draw (0.6,-1) rectangle (0.1,-0.6);
            \node at (0.35,-0.8) {$e_1$};
            \draw (-0.25,-0.6) arc(180:0:0.25);
            \draw (-0.45,-0.6) to[out=up,in=225] (0,-0.1) -- (0,0.1);
            \draw (0.45,-0.6) to[out=up,in=-45] (0,-0.1);
        \end{tikzpicture}
    \]
    can be written as a linear combination of the morphisms \cref{brutal} with the two leftmost top strands brought to the bottom using caps.  The coefficients in this linear combination all vanish \emph{if and only if} $\delta=26$.
\end{rem}

\begin{rem} \label{webs}
    Suppose we have an idempotent $e_2 \in \Fcat(\go^{\otimes 3}, \go^{\otimes 3})$ such that $\Kar(\Phi)(\go^{\otimes 3}, e_2) = V_{\omega_2}$.  Let $\Wcat = \Wcat(\alpha,\delta)$ be the full monoidal subcategory of $\Kar(\Fcat_{\alpha,\delta})$ generated by
    \begin{equation} \label{list}
        (\go^{\otimes 2}, e_1),\quad
        (\go^{\otimes 3},e_2),\quad
        (\go^{\otimes 2},e_3),\quad
        \go.
    \end{equation}
    Then morphisms in $\Wcat$ can be depicted as string diagrams with strands labeled by elements of $\{1,2,3,4\}$, with a strand labeled $i$ corresponding to the identity morphism of the $i$-th object in the list \cref{list}.  The category $\Wcat$ should be a degenerate (that is, $q=1$) web category of type $F_4$.
\end{rem}


\bibliographystyle{alphaurl}
\bibliography{F4}

\end{document}